\documentclass[11pt]{amsart}
\usepackage{amscd,amssymb,amsmath,enumerate}
\usepackage[all]{xy}
\usepackage{graphicx}
\usepackage{graphics}
\usepackage{latexsym}
\usepackage[all]{xy}
\usepackage{psfrag}
\xyoption{matrix} \xyoption{arrow}
\usepackage{parskip}

\newtheorem{prop}{Proposition}[section]
\newtheorem{thm}[prop]{Theorem}
\newtheorem{lemma}[prop]{Lemma}
\newtheorem{cor}[prop]{Corollary}

\newtheorem{definition}[prop]{{Definition}}

\newtheorem{remark}[prop]{{Remark}}

\newtheorem{Example}[prop]{Example}

\newcommand{\cX}{{\mathcal X}}

\newcommand{\cQ}{{Q}}

\newcommand{\cO}{{\mathcal O}}

\newcommand{\cS}{{\mathcal S}}
\newcommand{\K}{{ K}}

\renewcommand{\aa}

\newcommand{\ov}{\overline}

\newcommand{\soc}{\operatorname{soc}\nolimits}
\newcommand{\rad}{\operatorname{rad}\nolimits}

\renewcommand{\dim}{\operatorname{dim}\nolimits}

\newcommand{\val}{\operatorname{val}\nolimits}

\newcommand{\Qg}{{\mathcal Q}_{\Gamma}}

\newcommand{\bc}{Brauer configuration}
\newcommand{\bca}{\bc\ algebra}

\usepackage{color}
\definecolor{candyapplered}{rgb}{1.0, 0.03, 0.0}
\newenvironment{Sib}{\color{candyapplered}}

\makeatletter
\def\thm@space@setup{%
  \thm@preskip=0.7cm \thm@postskip=0.3cm
}
\makeatother

\begin{document}

\date{today}
\title[Special multiserial algebras and modules]
{Multiserial and special multiserial algebras and their representations}

\author[Green]{Edward L.\ Green}
\address{Edward L.\ Green, Department of
Mathematics\\ Virginia Tech\\ Blacksburg, VA 24061\\
USA}
\email{green@math.vt.edu}
\author[Schroll]{Sibylle Schroll}
\address{Sibylle Schroll\\
Department of Mathematics \\
University of Leicester \\
University Road \\
Leicester LE1 7RH \\
United Kingdom}
\email{schroll@le.ac.uk }
%\author[Snashall]{Nicole Snashall$^\ast$}\thanks{$^\ast$ Corresponding author}
%\address{Nicole Snashall\\
%Department of Mathematics \\
%University of Leicester \\
%University Road \\
%Leicester LE1 7RH \\
%United Kingdom}
%\email{njs5@leicester.ac.uk}
%\author[Taillefer]{Rachel Taillefer}
%\address{Rachel Taillefer\\ Clermont Universit\'e, Universit\'e Blaise Pascal, Laboratoire de Math\'ematiques\\ BP 10448, F-63000 Clermont-Ferrand -
%CNRS, UMR 6620, Laboratoire de Math\'ematiques, F-63177 Aubi\`ere\\
%France}
%\email{Rachel.Taillefer@math.univ-bpclermont.fr}

\thanks{This work was supported through the Engineering and Physical Sciences Research Council, grant number EP/K026364/1, UK and by the University of Leicester in form of a study leave for the second author. }

\begin{abstract}
In this paper we study multiserial and special multiserial algebras. These algebras are a natural generalization of biserial and special biserial algebras to algebras of wild representation type. 
We define a module to be  multiserial if its radical is the sum of uniserial modules
whose pairwise intersection is either 0 or
a simple module.  We show that all finitely generated modules over a special multiserial algebra are multiserial. In particular, this implies that, in analogy to special biserial algebras being biserial,  
special multiserial algebras are multiserial.  We then  show that 
the class of symmetric special multiserial algebras
coincides with the class of Brauer configuration algebras, where the latter are a  generalization of Brauer graph algebras. We end by showing 
that any symmetric algebra with radical cube zero is special multiserial and so, in particular, it is a Brauer configuration algebra.  
	\end{abstract}
\date{\today}
\maketitle

%\setcounter{tocdepth}{1}
%\tableofcontents

\section{Introduction}

In the study of the representation theory of finite dimensional algebras, the introduction of presentations of algebras by quiver and relations has 
led to major advances in the field.  Such presentations of algebras combined with another combinatorial tool 
that has proven powerful, the Auslander-Reiten translate and the 
Auslander-Reiten quiver, have led to many significant advances in the theory, to cite but a selection of these, see for example \cite{Bo, G, GR, R} or for an overview see \cite{ASS, SS1, SS2, ARS}. The addition of special 
properties such as semi-simplicity, self-injectivity, Koszulness, finite and tame representation type,  and finite global dimension, to name a few, have led 
 to structural results and classification theorems.  For example, the 
Artin-Wedderburn theorems for semi-simple algebras \cite{H}, 
the classification of 
hereditary algebras of finite representation type \cite{BGRS}, Koszul 
duality \cite{P}, classification of Nakayama algebras \cite{ARS}, covering theory of algebras  \cite{BG}, the study of tilted algebras \cite{BB, HR} and more recently, the study of cluster-tilted algebras beginning with \cite{BMR, CCS}. 

 Biserial and special biserial algebras have been the object of intense study at the end of the last century. Many aspects of the representation theory of these algebras is well-understood, for example, to cite but a few of the earlier results, the structure of the indecomposable representations \cite{Ri, GP, WW}, almost split sequences \cite{BR}, maps between indecomposable representations \cite{CB, K}, and the structure of the Auslander-Reiten quiver \cite{ESk}. Recently there has been renewed interested in this class of algebras. On the one hand this interest stems from its connecting with cluster theory.  In \cite{ABCP} the authors show that the Jacobian algebras of surface cluster algebras are gentle algebras, and hence a subclass of special biserial algebras. This class has been extensively studied since, see  \cite{CC, Ka} for examples of the most recent results. On the other hand with the introduction of $\tau$-tilting and silting theory \cite{AIR, AI}, there has been a renewed interested in special biserial algebras and symmetric special biserial algebras, in particular, see \cite{ AAC, MS, Z1, Z2}.

For self-injective algebras, Brauer tree and Brauer graph algebras have been useful in the classification of group algebras and blocks of group algebras of finite and tame representation type \cite{BD, Br, Hi, E} and the derived equivalence classification of self-injective algebras of tame representation type, see for example in  \cite{A, Sk} and the references within.  In these classifications  biserial and special biserial algebras have played an important role.

In this paper, we  study two classes of algebras, \emph{multiserial} and \emph{special multiserial} algebras introduced in \cite{VHW}, that are mostly of wild representation type. These algebras generalize biserial and special biserial algebras. In fact, they contain the classes of biserial and special biserial algebras and  we will see that they also contain the class  of symmetric algebras
with radical cube zero. One common feature of these classes is that their representation theory is largely controlled by the uniserial modules. The same is true for multiserial and special multiserial algebras.  

  We say that a module $M$ over some algebra is  \emph{multiserial} if the radical of $M$  is a sum of uniserial modules $U_1,\dots
U_l$ such that, if $i\ne j$, then $U_i\cap U_j$ is either $(0)$ or a 
simple module.  An algebra
$A$ is \emph{multiserial} if, as a right and left module, $A$ is multiserial. 	 For the definition of a \emph{special multiserial} algebra see Definition~\ref{def-specialmultiserial}. We note that the definition of  multiserial algebra as well as that of special multiserial algebra first appears in \cite{VHW}. Subsequently multiserial algebras and rings have been studied in \cite{KY} with a focus on hereditary multiserial rings, and with a slightly more general definition of multiserial algebra, they appear in \cite{J, M, BM}.

One of the main results of this paper is that any module  $M$ over a special multiserial algebra 
is multiserial.  As a consequence, a special multiserial algebra is a multiserial algebra, generalizing
the work of \cite{SW}  on special biserial algebras. Since special multiserial algebras are,
in general, wild, such a general result on the structure of modules is surprising.

\textbf{Theorem A.} \textit{Let $K$ be a field and let $A$ be a special multiserial $K$-algebra. Then every finitely generated $A$-module is a multiserial module.   }

 \textbf{Corollary} \textit{Any special multiserial algebra is a multiserial algebra.}

In Section \ref{arrowfree}, we introduce the concept of a ring having an
arrow-free socle.  We show that every self-injective finite dimensional algebra has an
arrow-free socle.  For an algebra with an arrow-free socle, we show that a
number of conditions are equivalent to
the algebra being special multiserial.

In \cite{GS}, Brauer configuration algebras were introduced. Their construction is based on combinatorial data, called a Brauer configuration, which generalizes Brauer graphs which in turn generalize Brauer trees. 
A Brauer configuration algebra is a finite dimensional symmetric algebra.
Our next result shows that, over an algebraically closed field,  
 Brauer configuration algebras and symmetric special multiserial algebras  coincide. 

\textbf{Theorem B.} \textit{Let $K$ be an algebraically closed field and let $A$ be a $K$-algebra. Then $A$ is a symmetric special multiserial algebra if and only if $A$ is a Brauer configuration algebra. } 

Another well-studied class of finite dimensional symmetric algebras is that 
of symmetric algebras with radical cube zero \cite{AIS, Be, ES1, ES2}. We prove that every symmetric algebra, over an algebraically closed field, with radical cube zero is a Brauer configuration algebra. 

\textbf{Theorem C.}  \textit{Let $K$ be an algebraically closed field. Then every symmetric $K$-algebra with Jacobson radical cube zero is a special multiserial algebra and, in particular, it is a Brauer configuration algebra. }
% and let $A$ be a symmetric $K$-algebra with Jacobson radical cubed zero. Then 
%$A$ is special multiserial and in particular, $A$ is a Brauer configuration algebra.}

In proving Theorems A, B and C, we obtain many structural results on multiserial and special multiserial algebras. 

%Given the close relation between symmetric radical cube zero algebras and (unoriented) graphs, see eg \cite{B, ES1, ES2}, one could as the question whether a symmetric radical cubed zero algebra is uniquely determined by the quiver associated to a graph. We show, by giving an  expclicit example, that this is not the case. Indeed we show that for a given quiver (that is the quiver of some symmetric radcal cubed zero algebra) there can be many non-isomorphic symmetric radical cube zero algebras with that same underlying quiver. 

The paper is outline as follows. In Section~\ref{section special multiserial}, we define multiserial modules, multiserial algebras and  special multiserial algebras. We show that a module over
a special multiserial algebra is multiserial. In Section~\ref{arrowfree} we define algebras with arrow-free socle and show properties of such algebras. Section~\ref{symmetricspecialmultiserial} is on symmetric special multiserial algebras and  we prove that an algebra is symmetric special multiserial if and only if it is a Brauer configuration algebra. Finally, in Section~\ref{radicalcubezero}, we show that symmetric
algebras with radical cube zero are special multiserial and hence, they are Brauer configuration algebras.

%%%%%%%%%%%%%%%%%%%%%%%%%%%%%%%%%%%%%%%%%%%%%%%%%%%%%%%%%%%%%%%%%%%%%%%%%%%%%%%%%%%%%%%%%%%%%%%%%%%%%%%%%%%%%%%%%%%%%%%%%%%%%%%%%%%%%%%%%

\section{Modules over special multiserial algebras}\label{section special multiserial}

%%%%%%%%%%%%%%%%%%%%%%%%%%%%%%%%%%%%%%%%%%%%%%%%%%%%%%%%%%%%%%%%%%%%%%%%%%%%%%%%%%%%%%%%%%%%%%%%%%%%%%%%%%%%%%%%%%%%%%%%%%%%%%%%%%%%%%%%%%

Let $K$ be  a field and let $A$ be a $K$-algebra. Unless explicitly stated otherwise, all modules considered are finitely generated right modules. Furthermore,  let $KQ/I $ be a finite dimensional algebra, for a quiver $Q$ and an admissible ideal $I$. Denote by $Q_0$ the set of vertices of $Q$ and by $Q_1$ the set of arrows in $Q$. By abuse of notation we sometimes view an element in $KQ$ as an element in $KQ/I$ if no confusion can arise. %The \emph{heart } of  an $A$-module $M$ is  given 
%by $\heart(M) = \rad(M) / \soc(M)$. 

Recall that a $K$-algebra $A$ is \emph{biserial} if for every indecomposable projective
left or right module $P$, there are uniserial left or right modules
$U$ and $V$, such that $\rad(P) = U+ V$ and $U \cap V $ is either zero or simple.
%the radical of every projective indecomposable $A$-module is a sum of at most two uniserial $A$-modules which intersect
% in at most one simple $A$-module. 

The algebra $A$ is 
\emph{special biserial} if it is Morita equivalent to an algebra of the form $K\cQ/I$ where $K\cQ$
is a path algebra and $I$ is an admissible ideal such that the following properties hold

\begin{itemize}
\item[(S1)] For every arrow $a$ in $\cQ$ there is at most one arrow $b$ in $\cQ$
 such that $a b\notin I$  and at most one arrow $c$ in  $\cQ$  such that $c a\notin I$.
\item[(S2)] At every vertex $v$ in $Q$ there are at most two arrows in $\cQ$ starting at $v$ and at most two arrows ending at $v$.
\end{itemize}

In particular, property (S2) implies that at every vertex there are at most two incoming and two outgoing arrows. 
A special multiserial algebra, as defined below, does not satisfy this property instead it only satisfies property (S1).

 We now give the definitions of  the two main concepts studied in this paper, namely multiserial algebras and special multiserial algebras (Definitions ~\ref{def-multiserial} and ~\ref{def-specialmultiserial}). These algebras were first defined and studied in \cite{VHW}.

\begin{definition}\label{def-multiserial}{\rm 
Let $A$ be a  $\K$-algebra.
% with Jacobson radical $\rad(A)$.  
 \begin{itemize}
\item[1)] We say that a left or right
$A$-module is \emph{multiserial} if  $\rad(M)$ can be written as
a sum of uniserial modules $U_1\dots U_l$ such that,   if $i\ne j $, then
$U_i\cap U_j$ either $(0)$ or a simple module. 
 
\item[2)]  
The algebra $A$ is \emph{multiserial} if $A$, as a left or right $A$-module
is multiserial.	
\end{itemize}} 
\end{definition}

%\ed{{\tt The next few sentences need to be modified!}}
%We remark that multiserial algebras have appeared in the literature, see for example \cite{KY, J} and \cite{M}. In each of these papers the definition of multiserial algebra is
%slightly different and they all vary  from the definition we give above. 
%However, as far as special multiserial algebras are concerned, 
%we are not aware of any previous such definitions. 

Furthermore, we remark that a multiserial algebra, that satisfies the additional property that  the radical is a sum of at most
two uniserial module whose intersection is zero or simple (on the left and
on the right)  is a biserial algebra.

 Recall from \cite{VHW} the definition  of special multiserial algebras. 

\begin{definition}\label{def-specialmultiserial} {\rm Let $A$ be a finite dimensional algebra.  
%  If $\pi\colon K\cQ\to  A$  is the canonical
%surjection and $x \in K\cQ$, we denote $\pi(c)$ by $\bar x$. 
We say that $A$  is a \emph{special multiserial algebra} if $ A$ is Morita equivalent to a
quotient $K\cQ/I$ of a path algebra $K\cQ$
 by an admissible ideal $I$ such that the following property holds

\begin{itemize}
\item[(M)]For every arrow $a$ in $\cQ$ there is at most one arrow $b$ in $\cQ$ such that 
$ab \notin I$ and at most one arrow $c$ in  $\cQ $ such that $ca \notin I$. 
\end{itemize}
%\[(S)\quad\quad \text{For every arrow }a \text{ {\Sib in }} \cQ \text{ there is at most one arrow }b \text{ {\Sib in }} \cQ
%\text{ such that }\]\[ a b\notin I \text{ and at most one arrow }c \text{ {\Sib in }} \cQ \text{ such that }  c a\notin I.
%\] 
}
\end{definition}

We note that the definition of an algebra being special multiserial 
is left-right symmetric. The following is the  main result of this section. 

\begin{thm}\label{multiserial theorem}
Let $K$ be a field, $A$ a special multiserial $K$-algebra, and $M$ a finitely generated  $A$-module. Then $M$ is multiserial. 
\end{thm} 

Theorem~\ref{multiserial theorem} has as an immediate consequence the following corollary.

\begin{cor}  Let $A$ be a special multiserial $K$-algebra. Then $A$ is a multiserial algebra.
\end{cor}

In the case that $A$ is a biserial algebra, we obtain as a Corollary to Theorem~\ref{multiserial theorem} the following result due to Skowro\'nski and Waschb\"usch.

\begin{cor}\cite{SW}\label{biserialisspecialbiserial} Let $A$ be a special biserial $K$-algebra. Then $A$ is a biserial algebra.  
\end{cor}

Before proving Theorem~\ref{multiserial theorem}, we present a series of Lemmas that we will use in its proof. We start with a very general Lemma.

\begin{lemma}\label{socle lemma}
Let $W, T, U, V$   be $\Lambda$-modules for a ring $\Lambda$ such that there is a commutative exact diagram 
$$\xymatrix{
0  \ar[r]   & W   \ar[r]^{\text{incl}}  \ar[d]  & T \ar[r]^f \ar[d]^g &   U \ar[r]  \ar[d]^= &   0 \\
0 \ar[r]     &\soc (V)   \ar[r]^{\text{incl}}  \ar[d] &  V \ar[r]      \ar[d]   & U   \ar[r]  &   0  \\  
& 0 & 0 \\
}$$
Then  $W \supseteq \soc(T)$. In particular, if $W$ is semisimple then $W = \soc(T)$. 
\end{lemma}

\begin{proof}
Let $S$ be a simple submodule of $T$. Suppose that $f(S) \neq 0$.
Then $g(S) \not\subseteq \soc(V)$. But $g(S)$ is a simple submodule of $V$, and hence in $\soc(V)$, a
contradiction.
\end{proof}

From now on let $A = KQ/I$ be a special multiserial algebra. For $a,b \in Q_1$ such that  the vertex at which $a$ ends   equals the vertex
at which  $b$ starts, we use the convention that $ab$ is the path starting with $a$ followed by $b$.   If $M$ is a right
$A$-module, we say that an element $m\in M$ is \emph{right uniform}, if there is some vertex  $v \in Q_0$ such that $m e_v = m$ where $e_v$ is the trivial path at vertex $v$. 

By the left-right symmetry of the definition of special multiserial  it is sufficient to  prove Theorem \ref{multiserial theorem} for right modules.  The
following result was proved, in the special case of modules of the form $aA$ where $a$ is an arrow, in \cite{VHW}.

\begin{lemma}\label{lem-uniser}
Let $M$ be an $A$-module and let $m \in M$ be right
uniform. For $a \in Q_1$, the  module generated by $ma$ is uniserial.
\end{lemma}

\begin{proof}
If $ma = 0$ the result follows. If $ma \neq0$ there exists at most one arrow $b$ such that $mab \neq 0$. If $mab \neq 0$, there exists at most one arrow $c$ such that $mabc \neq0$.    Continuing in this fashion we see that the submodule generated by $ma$ is uniserial. 
\end{proof}

We note that for a general algebra $\Lambda$ if $M$ is a  $\Lambda$-module such that  $\rad^2(M) = 0$  then $\rad(M)$, being semisimple implies
$M$ is a multiserial module.
%\ed{take out next sentence}We return to the study of modules for special multiserial algebras. 

\begin{lemma}\label{generators}
Let $M$ be an $A$-module with $\rad^2(M) \neq 0$. Then there exist right uniform elements  $u_1, \ldots, u_t$ in $\rad(M) \setminus \rad^2(M)$ such that 

1) $u_i = m_i a_i$ for some right uniform elements $m_i \in M \setminus \rad(M)$ and $a_i \in Q_1$, 

2) $u_i A$ is a uniserial module, 

3) $\sum_{i = 1}^t u_i A = \rad(M)$,

4) $\rad(M)/ \rad^2(M) \simeq \bigoplus \pi(u_i) (A / \rad(A))$ where $\pi: \rad(M) \to \rad(M)/ \rad^2(M)$ is the canonical surjection. 
\end{lemma}

\begin{proof}
 Choose a set of right uniform generators of $M$. Right multiplying these elements by arrows in $Q$ yields a generating set of   
$\rad(M)$ consisting of right uniform elements. 

Applying $\pi$ to this set we get a generating set of $\rad(M) / \rad^2(M)$, so we may select $u_1, \ldots, u_t$ so that $\pi(u_1), \ldots, \pi(u_t)$ is a minimal generating set of $\rad(M) / \rad^2(M)$ which is a
semi-simple module. 

Each $u_i = m_ia_i$ for some right uniform $m_i \in M \setminus \rad(M)$ and $a_i  \in Q_1$. By Lemma \ref{lem-uniser}
$u_i A$ is a uniserial module.  Parts (3) and (4) are clear from the
construction.
\end{proof}

Define the following partial order on the set of paths in $Q$. For $p, p'$ paths in $Q$, 
we say $p \geq p'$, if $p = q p'$  for some path $q$ in $Q$. The following Lemma follows directly from condition (M). 

\begin{lemma}\label{comparable lemma} Let $p$ and $p'$ be paths in $Q$ and $a\in Q_1$.
If $pa \neq 0$ and $p'a \neq 0$ then either $p \geq p'$ or $p' >p$. Hence $p$ and $p'$ are comparable. 
\end{lemma}

For $M$ an $A$-module, let $\ell(M)$ be the number of non-zero terms in a composition series of $M$. 
We introduce the set of short exact sequences satisfying the properties P1) - P3) which are defined below.  
Let

 $$\begin{array}{ll} \cS_M = \{ \varepsilon: 0 \to L \to \bigoplus_i U_i \to 
\rad(M) \to 0  \sloppy  \mid  & \varepsilon  \mbox{ is a short exact sequence}\\
& \mbox{satisfying P1) - P3)
  defined below} \}.
\end{array}$$

\begin{enumerate}
\item[P1)]  Each $U_i$ is a uniserial submodule of $\rad(M)$. 

\item[P2)] $\sum_i U_i = \rad(M)$.

\item[P3)] The map $\bigoplus U_i / \rad(U_i) \to \rad(M) /\rad^2(M)$ induced \sloppy from the map \linebreak  $\bigoplus U_i  \to \rad(M) $ is an isomorphism.
\end{enumerate}

Note that by Lemma~\ref{generators}, the set $\cS_M$ is not empty. Let \[\alpha = {\rm min}\{ \ell(L) \mid 0 \to L \to \bigoplus_i U_i \to \rad(M) \to 0 \in \cS \}\] and let 
$$\cS^*_M = \{ \varepsilon \in \cS_M \mid \ell(L) = \alpha \}.$$ 

 For a general ring $\Lambda$ and $\Lambda$-modules $N_1, \ldots, N_t$, the \emph{support of an element $n$}, where $n  = \sum  n_i \in \bigoplus_i N_i$, with $n_i \in N_i$, is the set of all $i$ such that $n_i\neq 0$ and is denoted by ${\rm supp}(n)$. The cardinality  $|{\rm supp}(n)|$ of  ${\rm supp}(n)$ is the number of components for which $ n_i \neq 0$. 

\begin{lemma}\label{support lemma}
Let $0 \to L \to \bigoplus_i U_i \to \rad(M) \to 0  \in \cS^*_M$, $x \in L$ with $x=  \sum \lambda_i u_i p_i$, $\lambda_i \neq 0 \in K$, $u_i \in U_i\setminus\rad(U_i)$  and $p_i$ a non-zero path in $A$, for all $i$. If for some $i$ and $j$ with $i \neq j$, $p_i \geq p_j$ then there exists, 
$0 \to L' \to \bigoplus_i U'_i \to \rad(M) \to 0  \in \cS^*_M$ such that the following 
diagram commutes 
$$\xymatrix{
0  \ar[r]   & L   \ar[r]  \ar[d]^g  & \bigoplus_i U_i  \ar[r] \ar[d]^f &   \rad(M) \ar[r]  \ar[d]^= &   0 \\
0 \ar[r]    & L'   \ar[r]   &  \bigoplus_i U'_i  \ar[r]         &  \rad(M)  \ar[r] &   0  \\  
}$$
where $g$ and $f$ are isomorphisms and where $| {\rm supp} (g(x)) | < | {\rm supp} (x) |$.
Note that if $x \notin \soc(L)$ then $g(x) \notin \soc(L')$.
\end{lemma}

\begin{proof}
Let $p_i \geq p_j$ for some  $i$ and $j$.  Then $p_i = q p_j$ for some path
$q$. Set $u'_l = u_l$ for all $l \neq j$. Then set $u'_j = u_j + \frac{\lambda_i}{\lambda_j} u_i q$. %Define a map $h: \bigoplus U'_i \to  \bigoplus U_i$ given by 
%$h( \sum \gamma_l u'_l p_l) = \sum_{l \neq j} \gamma_l u'_l p_l + \gamma_l(u_j + \frac{\lambda_i}{\lambda_j}) u_i$  
 Let $U_l'=u_l'A$ for all $l$.  It is immediate that  $\oplus_lU_l=\oplus_lU'_l$.
In fact, $ \sum \gamma_l u_l x_l =
\sum \gamma_l u'_l x_l - \gamma_j  \frac{\lambda_i}{\lambda_j} u'_i qx_j$ where $\gamma_l \in K$ and $x_l$ are paths.
Define a map $f: \bigoplus U_l \to  \bigoplus U'_l$ given by $f( \sum \gamma_l u_l x_l) =
\sum \gamma_l u'_l x_l - \gamma_j  \frac{\lambda_i}{\lambda_j} u'_iq x_j$ where $\gamma_l \in K$
and $x_l$ are paths. 
% Recall that $u'_l = u_l$ for $l \neq j$.
 Then $f$ just changes generating sets and hence is an isomorphism. Now $f(x) = f( \sum \lambda_l u_l p_l) = \sum \lambda_l u'_l p_l -   \lambda_i u'_i p_i  = \sum_{l \neq i} \lambda_l u'_l p_l $. Hence $| {\rm supp} (f(x)) | < | {\rm supp} (x) |$.
Since $g$ is a restriction of $f$, $x \in L$ and  $g(x) \in L'$, we have $| {\rm supp} (g(x)) | < | {\rm supp} (x) |$. The statement on the socles follows from the fact that $g$ is an isomorphism.
\end{proof}

Our final lemma is the following:

\begin{lemma}\label{lem: socle-lemma}  Suppose there exists a short
exact sequence
\[
0\to L\to\bigoplus_iU_i\to \rad(M)\to 0
\]
 in $\cS_M$, in particular, this implies that the $U_i$ are uniserial submodules of $M$ satisfying {\rm P1) - P3)} above. Suppose further that 
 $L$ is semisimple.   Then $M$ is multiserial.
\end{lemma}

\begin{proof}  
If $M$ is not a multiserial module, then for every choice of
uniserial modules $U_1,\dots, U_t$ such that $\sum_iU_i=\rad(M)$,
for some $i\ne j$,  $U_i\cap U_j$ is neither 0 nor a simple module.

But $U_i\cap U_j$ is isomorphic to a submodule of the
kernel of the canonical surjection $\bigoplus_iU_i\to \rad(M)$.
Hence,  since $U_i\cap U_j$ is a submodule of a uniserial
module,  it follows that $\rad(U_i\cap U_j)$ is nonzero.  Thus, the kernel of any
map $\bigoplus_iU_i\to \rad(M)$ is never semisimple.
The result follows. \end{proof}

\emph{Proof of Theorem~\ref{multiserial theorem}:} Without loss of
generality, we may assume that $M$ is indecomposable.
If $\rad^2(M) = 0 $, we have seen that the result is true. Assume that $\rad^2(M) \neq 0$. 
If there exists $0 \to L \to \bigoplus_i U_i \to \rad(M) \to 0 \in\ \cS^*_M$ such that $L$ is semi-simple 
then the result follows from Lemma \ref{lem: socle-lemma}.

Suppose no such $L$ exists, that is, $L$ is not semi-simple
for every short exact sequence in $\cS^*_M$. We will now show that this leads to a contradiction. Consider the set $$\cX = \{ x  \mid \mbox{ there is a s.e.s }  0 \to L \to \bigoplus_i U_i \to \rad(M) \to 0  \mbox{ in } \cS^*_M \mbox{ and } x \in L \setminus \soc(L)\}.$$
Let $\cX_{min} = \{ x \in \cX \mid |{\rm supp} (x)| \mbox{ is minimal in } \{{|\rm supp}(y)|, y \in \cX \}\}$. 

Let $x \in \cX_{min}$ where $0 \to L \to \bigoplus_i U_i \to \rad(M) \to 0$ and $x \in L \setminus \soc(L)$.   Then $x = \sum \lambda_i u_i p_i$
with $\lambda_i \in K$ and with $u_i \in U_i\setminus \rad(U_i)$. Note that by choice of $x$ the number of non-zero $\lambda_iu_ip_i $ is as small as possible.

By Lemma~\ref{support lemma} and minimality of $|{\rm supp}(x)|$, all the $p_i$ are  incomparable  in the partial order on the paths in $KQ$ defined earlier. Since $x \notin \soc(L)$, there exists an arrow $a \in Q_1$ such that $xa \neq 0$.

Since $xa \neq 0$ there exists $u_i p_i a \neq 0$ for some $i$. If  for some $j \neq i$, $u_j p_j a \neq 0$ then $p_ia \neq 0$ and $p_j a\neq 0$ and hence  by Lemma~\ref{comparable lemma}, $p_i$ is comparable to $p_j$.  Then by  Lemma~\ref{support lemma} we obtain a contradiction to the minimality of $|{\rm supp}(x)|$. Hence $xa = \lambda_i u_i p_ia$. So we have $xa \in L \cap U_i$  for exactly one $i$.

Now consider 
 $$\xymatrix{&& 0 & 0  && \\
(1) & 0  \ar[r]   & L/(L \cap U_i)   \ar[r]  \ar[u]  & U_i / (L \cap U_i) \oplus (\bigoplus_{j \neq i} U_j)   \ar[r] \ar[u] &   \rad(M) \ar[r]   &   0 \\
(2) & 0 \ar[r]    & L   \ar[r] \ar[u]  &  \bigoplus_i U_i  \ar[r]      \ar[u]   &  \rad(M)  \ar[r]  \ar[u]&   0  \\  
&& L \cap U_i \ar[r]^= \ar[u] &  L \cap U_i \ar[u] \\
&& 0 \ar[u] & 0 \ar[u] \\ 
}$$
Since the short exact sequence in line (2) is in $\cS^*_M$, clearly the one in line (1) is in $\cS_M$. But 
$\ell(L / (L \cap U_i) < \ell(L)$ thus contradicting the minimality of $L$ and so the assumption that $L$ is not semi-simple is false and the result follows. \hfill $\Box$

\subsection*{ More structural results on the uniserial modules over a special multiserial algebra.} Due to the importance of uniserial modules, we end this section with a few structural
results about such modules over a special multiserial algebra.   In Section~\ref{symmetricspecialmultiserial} we show that a symmetric special multiserial algebra is a Brauer configuration
algebra.  In this case, further  structural results on the uniserial
modules can be found in \cite{GS}.

Let $A=K\cQ/I$ be a finite
dimensional special multiserial
algebra,
let $J$ be the ideal in $K\cQ$ generated  by the arrows in $\cQ$, and let $N\ge 2$ be an
integer such that $J^N\subseteq I\subseteq J^2$. %Recall that $\rad(A) = J/I$. 
If $x\in K\cQ$, we denote its image in $ A$ 
  by $\bar x$. If $p$ is a path in $\cQ$, the \emph{length $\ell(p)$ of $p$} is the number of arrows in $p$.
 
For every arrow in $\cQ$, we now define a set of paths starting or ending with  that arrow. 
Let $a$ be an arrow in  $\cQ$ and let $i$ be a non-negative integer. We set $p_0(a) = a$ and we 
define $p_i(a) = aq$ where $q$ is the unique path in $\cQ$ of length  $i$ such that $aq \notin I$ if such a path
$q$ exists.  If no such path exists, then $p_i(a)$ is not defined.
  Thus $p_1(a) = ab$ for a unique arrow 
$b$ such that $ab \notin I$ if such an arrow $b$ exists and $p_2(a)  =abc$ for a unique arrows $b$ and $c$ such that $abc \notin I$
if such $b$ and $c$ exist. Now define $p_{-j}(a) = qa$ where $q$ is the unique path 
		of length $j$ in $\cQ$ such that $qa \notin I$ if such a $q$ exists. Again the uniqueness of $q$, if it exists, follows from (M). Let $t(a)$ denote the largest non-negative integer $i$ such that $p_i(a) \notin I$. Let $s(a)$ be the largest non-negative integer $j$ such that 
$p_{-j}(a) \notin I$. 
Note that $0\leq s(a), t(a) \leq N-1$. Furthermore, $\overline{p_{t(a)}(a)}$ is in the right socle of $A$ and 
$\overline{p_{-s(a)}(a)}$ is in the left socle of $A$. 

The next lemma provides a number of results about these paths. 

\begin{lemma}\label{lem-path}Let $A=K\cQ/I$ be a special multiserial $K$-algebra.
\begin{enumerate}
\item 
Suppose that $q,q'$ are paths in $K\cQ$ and $a$ is an arrow in $\cQ$ such that 
$qaq'\not\in I$. Then $qa=p_{-\ell(q)}(a)$ and $aq'=p_{\ell(q')}(a)$.
\item 
Suppose that $q=a_1a_2 \cdots a_{i-1}a_i a_{i+1}\cdots a_r$ and $q'=b_1\cdots b_{i-1}a_i b_{i+1}\cdots b_r$ are paths in $K\cQ$ such that
$q\not\in I$  %$1\le i\le r$ and $b_1,\dots,b_{i-1}, b_{i+1}, \dots, b_r$ are arrows and
 and $q\ne q'$. Then $q'\in I$.
\item  For $0\le i\le j\le t(a)$, $p_i(a)q=p_j(a)$, for some path $q$.
\item  For $0\le i\le j\le s(a)$  $qp_{-i}(a)=p_{-j}(a)$, for some path $q$.
\end{enumerate}\end{lemma}

\begin{proof}  The proof is just  repeated applications of condition (M).
\end{proof}

%Let $a$ be an arrow in $\cQ$ and let $U_a$ be the (right) 
%$A$-module  given by  the right ideal of $A$ generated by $a$. 
%Denote the set of arrows of $\cQ$ by  $\cQ_1$.

We now examine the structure of uniserial modules defined by an arrow in $Q$.

\begin{lemma}\label{lem-Ua}
Let $A = K\cQ / I$ be a special multiserial algebra and let $a$ be an arrow in $Q$. Set 
$U_a = aA$. Then 
\begin{enumerate}
\item The  $A$-module  $U_a$ is uniserial. 
\item The  $A$-module  $U_a /\soc(U_a)$  is  uniserial. 
\item We have $\soc(U_a) = U_a \cap \soc(A)$. 
\item We have $\rad(A) = \sum_{b \in \cQ_1} U_b$.  
\item The set $\{\overline{p_0(a)}, \overline{p_1(a)}, \ldots,  \overline{p_{t(a)}(a)}   \}$ is a $K$-basis of $U_a$. 
\end{enumerate}
\end{lemma}

\begin{proof}
Let $a \in \cQ_1$. If $U_a$ is a simple $A$-module, part (1) follows. Assume that $U_a$ is not a simple module. By condition (M), there is at most one arrow $b$ such that $ab \notin I$. It follows that $U_a / U_a \rad(A)$ and $U_a \rad(A) / U_a \rad^2(A)$ are both simple modules. 
Continuing in this fashion proves part (1). 

Part (2) follows from part(1). 

Part (3) follows from the observation that $p_{t(a)}(a)$ is in the right socle of $A$ and that 
it is also in the socle of $U_a$. 

Part (4) holds since $\sum_{b \in \cQ_1} U_b$ is the right submodule of $A$ generated by all arrows in $\cQ$. Hence $\sum_{b \in \cQ_1} U_b = J/I = \rad(A)$, where $J$ is the ideal in $K\cQ$ generated
by the arrows of $\cQ$.

We now prove part (5). It is clear that $\{\overline{p_0(a)}, \overline{p_1(a)}, \ldots,  \overline{p_{t(a)}(a)}   \}$ generates $U_a$. So we are left to show that if 
$\sum_{i=0}^{t(a)} \lambda_i p_i(a) \in I$, with $\lambda_i \in K$, then $\lambda_i = 0 $ for 
all $i$.  Suppose for contradiction that there is an integer $i$, $0 \leq i \leq t(a)$ such that 
$\lambda_i \neq 0$. Let $i_0$ be the smallest such $i$. 

By Lemma~\ref{lem-path} (4), $p_j(a) = p_{i_0}(a) q_j$, for $j \geq i_0$ and some path $q_j$ of length $j-i_0$ starting at the vertex, $w$, at which $p_{i_0}(a)$ ends.  
 Thus 
$$ \sum_{i=0}^{t(a)} \lambda_i p_i(a) = p_{i_0}(a)  \sum_{j = i_0}^{t(a)} \lambda_{j} q_j.$$
Let $e_w$ be the idempotent in $K \cQ$ associated to $w$. Then,  noting that if $j=i_0$, $q_j$ is
of length $0$ and hence
$q_j=e_w$,   we have
$$\sum_{j = i_0}^{t(a)} \lambda_j q_j = \lambda_{i_0} e_w + \sum_{j = i_0+1}^{t(a)} \lambda_{j} q_j.$$
Let $x  = \sum_{j = i_0+1}^{t(a)} \lambda_j q_j$. We have $x \in J$ and since $J^N \subset I$, there is some $y \in K \cQ$ such that 
$$ (\lambda_{i_0} e_w + x) y +I = e_w +I$$
 where $y$ is obtained as follows:
$ (\lambda_{i_0} e_w + x) (\lambda_{i_0}^{-1} e_w- x)= e_w-x^2$.
Then $(e_w-x^2) (e_w+x^2) =e_w-x^4$ and  continuing in this fashion we finally obtain 
$(\lambda_{i_0} e_w + x) (\lambda_{i_0}^{-1} e_w- x) (e_w+x^2)(e_w+x^4)\cdots =
e_w-x^{2n}$ and $x^{2n}\in I$.
 Hence 
$$ \sum_{i=0}^{t(a)} \lambda_i p_i(a)y + I = p_{i_0}(a) +I.$$
But by assumption $ \sum_{i=0}^{t(a)} \lambda_i p_i(a) \in I$ and hence $p_{i_0}(a) \in I$, 
a contradiction. This completes the proof. 
\end{proof}

\section{Algebras with arrow-free socles} \label{arrowfree}

In this section we introduce the concept of an algebra with arrow-free socle. We show that the socle of a self-injective algebra of radical series with length at least 3  is arrow-free. We
also show that for an algebra with arrow-free socle, condition (M) is equivalent to a stronger condition (M$'$) defined below. 

We fix the following notation.  We let $A=K\cQ/I$
be an indecomposable finite dimensional algebra with $I$ an admissible
ideal in the path algebra $K\cQ$.  Denote by $\pi: K\cQ \rightarrow A$ the 
canonical surjection and let $\bar{a} = \pi(a)$, for $a \in \cQ_1$.

\begin{definition}{\rm
We say that the socle of $A$ is \emph{arrow-free} if, for each $a \in \cQ_1$,  we 
have $\bar{a} \notin \soc (_{A}A)$ and   $\bar{a} \notin \soc (A_{A})$ 
where $_{A}A$ denotes the left $A$-module $A$ and $A_{A}$ the 
right $A$-module $A$.
}
\end{definition}

We first show that the socle of a self-injective algebra is arrow-free. 

\begin{prop}\label{lem-self-inj implies arrow-free}
Let $A$ be self-injective and $\rad^2(A) \neq 0$. Then the socle of $A$ is arrow-free.
\end{prop}

\begin{proof}
Suppose $\bar{a} \in \soc (A_{A})$ and suppose that $a$ is an arrow from 
a vertex $v$ to a vertex $w$ in $\cQ$. If $v = w$ then $A$ is isomorphic to $K[x]/(x^2)$ since $A$ is self-injective.   

Now suppose that $v \neq w$. If there is another arrow starting at $v$, by 
multiplying by arrows we would obtain a path $p \neq a$ such that $\bar{p}$ is a non-zero element in $\soc(A_A)$. Since $A$ is self-injective we get $
p-\lambda a \in I$, for some non-zero $\lambda \in K$, contradicting 
that $I$ is admissible. Thus $a$ is the only arrow starting at $v$. Suppose that $b$ is an arrow ending at $v$. If $ba \notin I$ then $\overline{ba} \in\
\soc(A_A)$ and hence, for some $\lambda \neq 0$, we have $ba - \lambda a \in I$, which contradicts that $I$ is admissible. Thus $ba \in I$ and $\bar{b} 
\in \soc(A_A)$. Continuing in this fashion, since $A$ is indecomposable,  we see that every arrow is in $\soc(A_A)$ and thus $\rad^2(A) =0$.
\end{proof}

We note that the converse does not hold in general. 

The following Lemma follows immediately from the definition of arrow-free. 

\begin{lemma} \label{lem- arrow-free implies M'}
If the socle of $A$ is arrow-free then for all arrows $a$ in $\cQ$ there are arrows 
$b$ and $c$ such that $ab \notin I$ and $ca \notin I$.   
\end{lemma}

From condition (M) it follows that understanding the paths of length 2 is crucial. Set 
$$\Pi = \{ ab | a, b \in \cQ_1, ab \notin I\}.$$

We say that a cycle $C$ is \emph{basic} if there are no repeated arrows
in $C$. We say that  a set  $\{C_1, \ldots, C_r \}$ of basic cycles is \emph{special} if the following conditions hold
\begin{enumerate}
\item for each arrow $a$ in $\cQ$, $a$ occurs in exactly one $C_i$, $i = 1, \ldots r$,
\item  The path $ab$ is in $\Pi$  if and only if $ab$ is a subpath of  some cycle  $C_i=c_1 \ldots c_n$  where we consider $c_nc_1$ a subpath of $C_i$. 
\end{enumerate}

We show that for an algebra with arrow-free socle the following condition is equivalent to condition (M).  Set condition

 \begin{itemize}
\item[(M$'$)]For every arrow $a$ in $\cQ$ there exists exactly one arrow $b$ in $\cQ$ such that 
$ab \notin I$ and exactly  one arrow $c$ in  $\cQ $ such that $ca \notin I$. 
\end{itemize}

\begin{prop}\label{prop-(M) (M') equivalence}
Let $A = KQ/I$ be a finite dimensional indecomposable algebra with $I$ an admissible ideal of $KQ$. Suppose that the socle of $A$ is arrow-free. Then the following are equivalent
\begin{enumerate}
\item Condition (M) holds, that is, $A$ is  special multiserial. 
\item Condition ($M'$) holds. 
\item The map $\varphi: \Pi \rightarrow \cQ_1$ given by $\varphi(ab) = a$ is  bijective. 
\item  The map $\psi: \Pi \rightarrow \cQ_1$ given by $\psi(ab) = b$ is bijective.
\item There exists a special set of cycles.
\end{enumerate}
\end{prop}

\begin{proof} The implication (1) implies (2) follows from Lemma~\ref{lem- arrow-free implies M'}.

To see that (2) implies (3), let $ab \in \Pi$. Then by (M$'$), $b$ is unique and hence $\varphi$ is one-to-one and well-defined. Again by (M$'$), given $a \in \cQ_1$ there exists  $b \in \cQ_1$ such that $ab \in \Pi$ and hence $\varphi$ is onto. 

For the implication (3) implies (2), let $a \in \cQ_1$. By (3) there exists a unique arrow $b$ such that $ab \notin I$. Now by Lemma~\ref{lem- arrow-free implies M'} there exists an arrow $c_1$ such that $c_1a \notin I$. Again by Lemma~\ref{lem- arrow-free implies M'} there exists an arrow $c_2$ such that $c_2c_1 \notin I$. Continue in this way until the first repeat of an arrow, that is, we have some path $c_n \ldots c_s \ldots c_1 c_0$ where $ c_0=a$. If the first repeat is $c_n = c_s$, we show that $s = 0$. If not then since $\varphi$ is bijective, we have $c_{n-1} = c_{s-1}$, a contradiction. It now follows that $c_1$ is unique and (2) follows. 

That (2) is equivalent to (4) is similar to the equivalence between (2) and (3). 

Next we show that (2) implies (5). Let $a = a_0 \in \cQ_1$. Then there
exists a unique $a_1 \in \cQ_1$ such that $a_0 a_1 \notin I$ and there exists a unique  $a_2 \in \cQ_1$ such that $a_1 a_2 \notin I$. Continue in this way until the first repeat of an arrow to obtain a sequence of arrows 
$a_0 \ldots a_n$. As above we have $a_n = a_0$. So  $a_0 \ldots a_{n-1}$  is a basic cycle $C_1$. If there is some arrow 
$b_0$ such $b_0 \neq a_i$, for $i = 1, \ldots , n$, then continue in the same fashion to obtain a cycle $C_2 = b_0 \ldots b_m$. By uniqueness, no $b_i = a_j$. Either all the arrows occur in $C_1$ and $C_2$ or we can continue this proccess and construct a $C_3$. Eventually one obtains a special set 
$\{C_1, \ldots, C_r\}$ of special cycles. 

Finally we prove that (5) implies (1). Let  $a \in \cQ_1$. By the definition of  a special set of cycles $\{C_1, \ldots, C_r\}$, there exists an $i$ such that $a \in C_i$. The second part of the definition of a special set of cycles implies that there exists unique arrows $b$ and $c$ such that $ab \notin I$ and $ca \notin I$.  
\end{proof}

\begin{remark}\label{rmk-arrow-free}
{\rm 
(1) The above Proposition does neither assume that the algebra is self-injective nor that it is special multiserial. 

(2) Suppose that $A$ is special multiserial and arrow-free. If there are paths $p, q$ in $K\cQ$ with $\ell(p) \geq \ell(q)$  and $a \in \cQ_1$ such that $pa \notin I$ and $qa \notin I$ then there exists a unique path $r$ such that $rq = p$. 
}
\end{remark}

\section{Symmetric special multiserial algebras and Brauer configuration algebras}\label{symmetricspecialmultiserial}

In this section we study special multiserial algebras that have the additional property of being symmetric algebras. In the case of symmetric special biserial algebras,  it is proved in 
\cite{Roggenkamp, S} that 
the class of symmetric special biserial algebras coincides with the class of  Brauer graph algebras.  We will show in this section that an analogous results holds for  symmetric special multiserial algebras. Namely, the main result of this section 
is   to show that the class of symmetric special multiserial algebras coincides with the class of Brauer configuration algebras.
 Brauer configuration algebras have been defined in \cite{GS} and they can be seen as generalizations  of Brauer graph algebras. We will recall their definition below.   %Brauer configuration algebras, defined in \cite{GS}, are a generalization of  Brauer graph algebras. %Indeed, we show that  a basic indecomposable  finite dimensional
%symmetric algebra $A$ is special multiserial  if and only if it is isomorphic to
%an indecomposable  \bca.  
 Note that in the present paper, we assume all Brauer configurations to be reduced.

 Before recalling definitions and further analysing the structure of symmetric special multiserial algebras, we first state the main result of this section.

\begin{thm}\label{thm-sma} 
Let $A=K\cQ/I$ be an indecomposable finite dimensional algebra over an algebraically closed field $K$
such that $I$ is an admissible ideal and  $ \rad(A)^2\ne 0$. Then $A$ is a symmetric special multiserial
algebra if and only if $A$ is a Brauer configuration algebra.
%
%%Let $A$ be a basic indecomposable symmetric special multiserial
%%$K$-algebra over  an algebraically closed field $K$.    Then $A$ is isomorphic to a
%\bca. Conversely, if $A$ is a Brauer configuration algebra then it is special multiserial. 
\end{thm}

%*******{\tt all new for a while} *****

\subsection{Definition of Brauer configuration algebras} We recall from \cite{GS} the definition of a (reduced) Brauer configuration algebra.  We start 
with the definition of a \bc , which generalizes a Brauer graph.   A \emph{\bc\ } $\Gamma$ is a
tuple $(\Gamma_0,\Gamma_1,\mu,\mathfrak o)$,  where
\begin{enumerate}
\item $\Gamma_0$  is a finite set of elements called \emph{vertices}.
\item $\Gamma_1$  is a finite collection of finite multisets of vertices which are called
\emph{polygons}.  Recall that a multiset is a set where elements can occur multiple times.
\item $\mu\colon \Gamma_0\to \{1,2,3,\dots\}$ is a set function called the \emph{multiplicity function}.
\item A vertex $\alpha$ is called \emph{truncated} if it occurs once in exactly one polygon
and $\mu(\alpha)=1$.
The sum over the polygons $V\in\Gamma_1$ of the number of times a vertex $\alpha$ occurs in $V$
is denoted $\val(\alpha)$.
 We say $\mathfrak o$ is an \emph{orientation} which means that, for each nontruncated vertex $\alpha$, there is a chosen 
cyclic ordering of the polygons  that contain $\alpha$, counting repetitions.  See the example and
the discussion below. 
\end{enumerate}
 We require that $\Gamma=(\Gamma_0,\Gamma_1,\mu,\mathfrak o)$ satisfies 
\begin{enumerate}
\item[C1.]  Every vertex in $\Gamma_0$ is a vertex in at least
one polygon in $\Gamma_1$.
\item[C2.]  Every polygon in $\Gamma_1$ has at least two vertices.
\item[C3.]  Every polygon in $\Gamma_1$ has at least one vertex
$\alpha$ such that  $\val(\alpha)\mu(\alpha)>1$.
\item[C4.]  If $\alpha$ is  a vertex in polygon $V$ and $\val(\alpha)\mu(\alpha)= 1$, that is, $\alpha$
is truncated then $V$ is a 2-gon.

\end{enumerate}

We note that C4 does not occur in the definition of a \bc\ in \cite{GS}.
In that paper a \bc\ was called \emph{reduced} if it satisfied C4.  In
this paper, all \bc s are ``reduced''.

\begin{Example}\label{ex1}{\rm
We give an example of a \bc.  Let $\Gamma_0=\{1,2,3,4\}$,
$\Gamma_1=\{V_1, V_2, V_3\}$ where $V_1=\{1,1,2,3,3\},
V_2=\{2,2,3\}$ and $V_3=\{2,4\}$, and
$\mu(i)=1$, except that $\mu(1)=3$ and $\mu(4)=2$.   To give an orientation, for each nontruncated
vertex, we need to be given a cyclic order of
the polygons that contain the vertex.  If a vertex occurs in a polygon more
than once, we will
use superscripts to denote these occurrences. Thus for vertex 1,
we need to order $V_1^{(1)},V_1^{(2)}$, for vertex 2, we must order $V_1,V_2^{(1)},V_2^{(2)},V_3$, for vertex
3 we must order $V_1^{(1)},V_1^{(2)},V_2$, etc.   So for vertex 1 we must have $V_1^{(1)}<V_1^{(2)}$, and to make it cyclic, we implicitly have $V_1^{(2)}<V_1^{(1)}$.
For vertex 2, there are  many choices of cyclic orderings, and for example, we will use
$V_1<V_3<V_2^{(1)}<V_2^{(2)} $, and to make it cyclic, we implicitly have
$V_2^{(2)}<V_1$.   Note that equivalently we could have taken any cyclic permutation of $V_1<V_3<V_2^{(1)}<V_2^{(2)} $.
%Equivalently, we could have taken any cyclic permutation;
%for example, $V_1<V_2^{(2)}<V_3<V_2^{(1)} $ or
 %$V_2^{(2)}<V_3<V_2^{(1)}<V_1 $.
 For vertex 3, take $V_1^{(1)}<V_1^{(2)}<V_2$ or a cyclic permutation of
this;  vertex 4, since $\mu(4)=2$ is not
truncated and we take the cyclic ordering to be just $V_3$ (with implicitly
$V_3<V_3$). }
\end{Example}

We call $V_{i_1}<V_{i_2}<\cdots
V_{i_m}$ a \emph{successor sequence of $V_{i_1}$ at  $\alpha$}
if $\alpha$ is a vertex in $ \Gamma_0$ and $V_{i_1}<V_{i_2}<\cdots
V_{i_m}$ is a cyclic ordering, obtained from the orientation $\mathfrak o$, of the polygons
containing $\alpha$ as an
element.   In this case, we say the $V_{i_{j+1}}$ is the successor of $V_{i_j}$ at $\alpha$,
for $j=1,\dots, m$ where $V_{i_{m+1}}=V_{i_1}$.                     

Fix a field $K$.
We now define the \bca\, $A$, associated to a \bc\ $\Gamma=(\Gamma_0,\Gamma_1,
\mu,\mathfrak o)$ via a quiver with relations.  That is, we will define a quiver 
$\cQ$ and a set of relations $\rho$ in the path algebra $K\cQ$ such that
$A$ is isomorphic to $K\cQ/I$, where $I$ is the ideal generated by $\rho$.
The vertex set of $\cQ$ is in one-to-one correspondence with $\Gamma_1$, the set
of polygons of $\Gamma$.  If $V$ is a polygon in $\Gamma_1$, we will denote the
associated vertex in $\cQ$ by $v$.  If  the polygon $V$ is a successor to the polygon
$V'$ at $\alpha$, there is an arrow from $v$ to $v'$, where $v$ is the vertex
in $\cQ$ associated to $V$
and $v'$  is the vertex in $\cQ$ associated to $V'$ in $\cQ$.   This gives a 
one-to-one correspondence between the set of successors in $\Gamma$ and the arrow set
in $\cQ$.

\begin{Example}\label{ex2}{\rm
The quiver $\cQ$ of Example \ref{ex1} is

\hskip 0.5in \begin{center} \includegraphics[scale=.7]{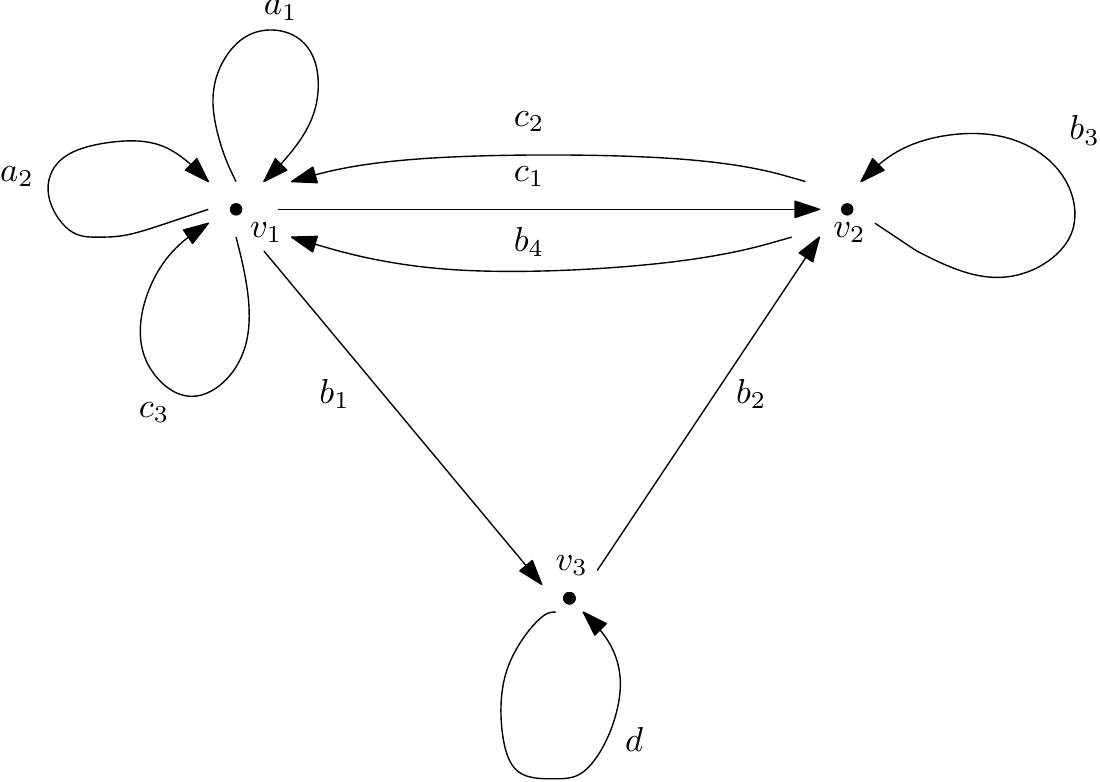}
\end{center}

\hskip 0.3in

Here $a_1$ corresponds to $V_1^{(2)}$ being a successor of $V_1^{(2)}$ and
$a_2$ corresponds to $V_1^{(1)}$ being a successor of $V_1^{(2)}$.  The arrows
labelled $b_1,b_2,b_3,b_4$ correspond to the successor sequence at vertex 2 of $\Gamma$.
The arrows
labelled $c_1,c_2,c_3$ correspond to the successor sequence at vertex 3 of $\Gamma$.
Finally, the arrow labelled $d$ corresponds to the successor sequence at vertex 4 of $\Gamma$. 
}\end{Example}

Before we define a generating set for the ideal of relations of $A$, we 
introduce the definition of a special $\alpha$-cycle at $v$, for
a nontruncated vertex $\alpha$ in $\Gamma_0$ and vertex $v\in{(\cQ_{\Gamma})}_0$.
Let
$ V_1< V_2 <\ldots < V_{\val(\alpha)}$ be the successor sequence
of $V_1$ at $\alpha$. For each $j$,
 let  $C_{j} = a_{j} a_{j+1}
\cdots a_{\val(\alpha)} a_1 \cdots a_{j-1}$ be the cycle in $\Qg$,   where the arrow $a_r$ corresponds to
the polygon $V_{r+1}$ being the successor of the polygon $V_r$ at the
vertex $\alpha$. 
Let $V$ be the polygon in $\Gamma_1$ associated to the vertex
$v\in{(\cQ_{\Gamma})}_0$ and suppose that 
$\alpha$ occurs $t$ times in $V$, with $t\ge 1$. Then there are $t$ indices $i_1, \dots, i_t$ such that $V = V_{i_j}$. 
We define
the \emph{special $\alpha$-cycles at $v$} to be  the cycles $C_{i_1}, \dots, C_{i_t}$.  We remark that
each $C_{i_j}$ is a cycle in $\cQ_{\Gamma}$, beginning and ending
at the vertex $v$. Note that if  $\alpha$ occurs only  once
in $V$ , then there is only one special $\alpha$-cycle at $v$. Furthermore, if  $V'$ is a polygon
consisting of $n$ vertices, counting repetitions, then there are a total of $n$
different special ${\alpha}'$-cycles at $v'$, for ${\alpha}'\in{(\cQ_{\Gamma})}_0$.  We will sometimes write \emph{special $\alpha$-cycle} (omitting the vertex) or simply \emph{special cycle} omitting both $\alpha$ and $v$, when
no confusion can arise.

\begin{Example}{\rm   Continuing the example, the special $ 1$-cycles at $v_1$ are
$a_1a_2$ and $a_2a_1$, the special  $2$-cycle at $v_1$ is $b_1b_2b_3b_4$, the special  $3$-cycles
at $v_1$
are $c_1c_2c_3$, and $c_3c_1c_2$, and there are no special  $4$-cycles at $v_1$.
There are no  special $1$-cycles at $v_2$, the  special $2$-cycles at $v_2$ are $b_3b_4b_1b_2$ and
$b_4b_1b_2b_3$, the special  $3$-cycle at $v_2$ is  $c_2c_3c_1$, and there are no  special $4$-cycles at $v_2$.  Finally there are no  special $i$-cycles at $v_3$ for  $i =1,3$.
The  special $2$-cycles at $v_3$ is $b_2b_3b_4b_1$ and the special 4-cycle at
$v_3$ is $d$.
}\end{Example}

There are three types of relations forming the generating set of relations $\rho_\Gamma$.

\textit{Relations of type one.}  For each polygon $V =\{\alpha_1, \ldots, \alpha_m\} \in \Gamma_1$ and each pair
of  nontruncated vertices $\alpha_i$ and $\alpha_j$  in $V$,  $\rho_\Gamma$ contains all relations of the form 
 $C^{\mu(\alpha_i)} - (C')^{\mu(\alpha_j)}$ or $(C')^{\mu(\alpha_j)} - C^{\mu(\alpha_i)}$ where $C$ is a special  $ \alpha_i$-cycle at $v$ and $C'$ is  a   special $ \alpha_j$-cycle at $v$.
%These are the type one relations. 

\textit{Relations of type two.} The type two relations are all paths of
the form $C^{\mu(\alpha)}a_1$ where 
$C=a_1\cdots a_m$ is a  special $\alpha$-cycle for some vertex 
$\alpha$.

\textit{Relations of type three.} These relations are
quadratic monomial relations of the form $ab$ in $K\cQ_{\Gamma}$ where $ab$
is not a subpath of any cycle $C$ where $C$ is a special
cycle.

 \begin{definition}{\rm  Let $K$ be a field and
 $\Gamma$ a Brauer configuration.  The \emph{Brauer configuration
algebra $\Lambda_{\Gamma}$ associated to $\Gamma$}  is defined to be 
 $K\cQ_{\Gamma}/I_\Gamma$, where  $\cQ_{\Gamma}$ is
the quiver associated to $\Gamma$ and $I_\Gamma$ is the ideal in  $K\cQ_{\Gamma}$ generated by the set of relations $\rho_{\Gamma}$ 
of types one, two and three.}
% If $\cQ_{\Gamma}$ is
%the quiver associated to $\Gamma$ and $\rho_{\Gamma}$ is the
%set of all relations of types one, two and three,  then the \emph{Brauer configuration
%algebra associated to $\Gamma$}, $\Lambda_{\Gamma}$, is defined to be
%$K\cQ_{\Gamma}/I$, where $I$ is the ideal in
%$K\cQ_{\Gamma}$ generated by $\rho_{\Gamma}$.
%}
\end{definition}

\begin{Example}{\rm
Continuing our example, some type one relations are $(a_1a_2)^3-(a_2a_1)^3,\linebreak (a_1a_2)^3-b_1b_2b_3b_4,
(a_2a_1)^3-b_1b_2b_3b_4, (a_1a_2)^3-c_1c_2c_3,  c_1c_2c_3-c_3c_1c_2,
b_2b_3b_4b_1-d^2$.
Some type two relations are $(a_1a_2)^3a_1, (a_2a_1)^3a_2, b_1\cdots b_4b_1, c_1c_2c_3c_1$,
and $d^3$.   Some type three relations are any $a_ib_j$, $b_ja_i$, $a_ic_k, a_id, db_j$, for
$1\le i\le 2, 1\le j\le 4, 1\le k\le 3$.  The set of relations of types one, two, and three generate the 
ideal of relations, but this set contains a large number of redundant relations and the set
is usually not a minimal generating set    for the ideal of relations.
}\end{Example}

% \subsection{{\Sib Proof of Theorem~\ref{thm-sma}}}
%always assume that all algebras considered,
%\ed{other than path algebras,} are finite dimensional and indecomposable. 
%that $A$  is a symmetric special multiserial algebra with $\rad(A)^2\ne 0$.

\subsection{ Properties of symmetric special multiserial algebras} 
Before proving Theorem~\ref{thm-sma}, we analyse the structure of  symmetric special multiserial algebras in more detail. For this we do not
necessarily need to assume that $K$ is algebraically closed. 

From now on and for the remainder of Section 4, in addition to our previous assumptions, we always assume that all $K$-algebras considered are indecomposable. In fact we assume for the 
rest of this subsection, unless otherwise stated, that $A = KQ/I$ is an indecomposable symmetric special multiserial algebra where  $I$ is an admissible ideal and where $ \rad(A)^2\ne 0$. 
%Then if   $\pi\colon K\cQ\to A$ such that $\ker(\pi) = I$,
%we identify  the arrows in $\cQ$ with their images in $A$ under the  map $\pi$.
%{\tt  See remark below with respect to these images... } 

\begin{remark} {\rm Since $A$ is symmetric, it is self-injective, and therefore by Proposition~\ref{lem-self-inj implies arrow-free} its socle is arrow-free. So in particular, the condition (M$'$) holds for A.} 
%Note that 
%if $A=K\cQ/I$ is an indecomposable symmetric finite dimensional algebra such that  $ \rad(A)^2\ne 0$ and 
%if $a$ is the image in $ A$ of an arrow in $\cQ$ 
%then $a\not\in\soc(A)$.   To see this, suppose $a \in\soc(A)$.  By admissibility of the ideal $I$ and symmetry
%of the algebra, we
%see that $a$ is the image of a loop at a vertex and that loop and vertex comprise
 %an isolated component
%of $\cQ$.  This either contradicts that $A$ is indecomposable or that $\rad(A)^2\ne 0$.  Hence $A$ is not isomorphic to $K[x]/(x^2)$. See the proof of Lemma \ref{lem-ll} for a more detailed
%proof.}
\end{remark}
 
%If a $K$-algebra $A$ is special multiserial and symmetric then property (M) implies a stronger property (S$'$). 

%{\Sib For the rest of  this subsection, unless otherwise stated, let $A = KQ/I$ be an indecomposable symmetric special multiserial algebra. }

%If a $K$-algebra $A$ is special multiserial and symmetric then property (M) implies a stronger property (M$'$). 
For the special multiserial algebra $A$, the additional property of being symmetric implies the existence of a permutation on the set of arrows of $Q$. 
Namely, let 
%  $A = KQ/I$ be an indecomposable symmetric special multiserial algebra
%with canonical surjection
 $\pi\colon K\cQ\to A$ be the canonical surjection. Suppose that  $a$ is an arrow in $\cQ$.  Since
$\rad(A)^2\ne 0$ and $A$ is symmetric, 
$\pi(a)\not\in \soc(A)$ and hence there
is some arrow $b$ such that $a b\notin I$.    By condition (M), $b$ is
unique. This leads to the following definition.
%{\Sib {\tt I'm not so happy with the paragraphs above, since we seem to repeat things a lot or introduce one notation only to then switch to another. For example, we say that we write $a$ for the image of an arrow $a$ of $Q$ in $A$. At the same time we introduce (several versions of) 
%$\pi$ and then use $\pi(a)$...it's not very conisisten and confusing. Maybe we should go over it together on Skype and decide on one notation}}. 

\begin{definition}{\rm 
  Let $A = KQ/I$ be an indecomposable symmetric special multiserial algebra where  $I$ is an admissible ideal and such that $ \rad(A)^2\ne 0$.
Letting $\sigma(a)$ denote the unique arrow such that $a\sigma(a)\notin
I$, 
 the assignment $\sigma: \cQ_1 \to \cQ_1$ given by  $a \mapsto \sigma(a)$ defines a permutation on the set of arrows $Q_1$ of $Q$. We  call it the \emph{permutation induced by $I$.} }
\end{definition}

We remark that a similar contruction has been observed by S.\ Ladkani in the 
context of Brauer configuration algebras.

Let $\sigma$ be the permutation induced by $I$.
Note that if $b\ne a$, then $\sigma(a)\ne \sigma(b)$ since there is at most
one arrow $c$ such that $c\sigma(a) \notin I$. Thus  $\sigma$ is bijective and
 $\sigma^{-1}(a)$ is the unique arrow such that $\sigma^{-1}(a) a\notin I$. 
Since
$\sigma$ is an isomorphism, $\cQ_1$ is partitioned into the orbits of $\sigma$, which
we denote by  $\{O_1,\dots,O_m\}$.   These orbits will play an important role in
what follows.  Note that if $O$ is an orbit and $a$ is an arrow in $O$,
then $O=\{a,\sigma(a),\sigma^2(a),\dots,\sigma^{s}(a)\}$ where the cardinality $|O|$ of $O$
is $s+1$.

Since $A$ is symmetric, by definition, there is a  non-degenerate $K$-linear
form $f \colon A\to K$ such that, for all $x,y\in A$, $f(xy)=f(yx)$ (that is, the form $f$ is symmetric) and
$\ker(f)$ contains no nonzero two-sided ideals of $A$.  Furthermore,  the fact that $A$ is symmetric implies
that $A$ is  self-injective, and  (since $A$ is basic) the left socle of $A$ is equal to the right socle of $A$ and 
they are both equal to  the two-sided
socle of $A$.  Note that $\ker(f)$ contains no nonzero two-sided ideals if and
only if $\ker(f)\cap\soc(A)= \{0\}$.

 We now prove a technical lemma about the
orbits of $\sigma$.

 %For each arrow $a\in\cQ_1$, if $a$ is in the $\sigma$-orbit $O$, let
%$c_a=a\sigma(a)\sigma^2(a)\cdots \sigma^{|O|-1}$.  
\begin{lemma}\label{lem-orbits} Let $A = K \cQ/I $ be an indecomposable \sloppy
%n indecomposable 
symmetric special 
multiserial $K$-algebra with $\rad(A)^2\ne 0$ and let $\sigma$ be the permutation on $Q$ induced by $I$. Given a $\sigma$-orbit  $O$ and an  arrow   $a$  in $O$, we set
 $$c_a=a\sigma(a) \sigma^2(a) \dots\sigma^{|O|-1}(a).$$ Then the following hold:
\begin{enumerate}
\item  The path $c_a$ is a cycle in $\cQ$.
\item The paths   $c_{\sigma^i(a)}$ are cycles in $\cQ$, for $0\le i\le |O|-1$.
\item There is an integer $m_a >0$ such that  $\overline{c_a^{m_a}}$ is a nonzero element in $\soc(A)$.
%\item There is an integer $m_a$ such that $c_a^{m_a}$ is a nonzero element in $\soc(A)$.
\item We have $m_a=m_{\sigma^i(a)}$ and 
$\overline{(c_{\sigma^i(a)})^{m_a}}$ is a nonzero element in $\soc(A)$, for $0\le i\le |O|-1$.
%\item $m_a=m_{\sigma^i(a)}$, for $0\le i\le |O|-1$, and 
 %$(c_{\sigma^i(a)})^{m_a}$ is a nonzero element in $\soc(A)$, for $0\le i\le |O|-1$.
 \item We have  $f(\overline{c_a^{m_a}})=f(\overline{(c_{\sigma^i(a)})^{m_a}})$, for $0\le i\le |O|-1$, where $f$ is the symmetric linear form defined above.

\end{enumerate}

\end{lemma}

\begin{proof}
By definition,  $\overline{\sigma^i(a)}\overline{\sigma^{i+1}(a)}\ne 0$, for $1 \leq i \leq  |O|-1$. Therefore
$c_a$  is a path in $\cQ$.  Since $\sigma^{|O|}(a)=a$, the arrow $\sigma^{|O|-1}(a)$ ends at the same vertex at which $a$ starts.
 Thus $c_a$ is a cycle and
(1) is proved.

Statement (2) follows from (1) by replacing $a$ by $\sigma^i(a)$.  

  By the definition of $\sigma$, there must be an integer $s$  such that
$\ov{a {\sigma(a)}\cdots {\sigma^{s-1}(a)}}$  is a non-zero element in $\soc(A)$.
Since $\sigma^{|O|}(a)=a$, there exists an integer $m_a$ such that
\[{a}{\sigma(a)}\cdots \sigma^{s-1}(a)
=(c_a^{m_a-1})
a\cdots\sigma^i(a),\]
for  $1\le i\le |O|-1$.  Now there are two cases. Firstly, if $i=|O|-1$, then (3) directly follows. Secondly, 
suppose $i<|O|-1$. We will show that this is not possible since it leads to a contradiction, thus proving
(3).   So if $i<|O|-1$ then $f(\ov{(c_a^{m_a-1})
a\cdots\sigma^i(a)})\ne 0$.   But  $\ov{\sigma^i(a)a}=0$  since $a\ne \sigma^{i+1}(a)$.
Thus  $\ov{ \sigma^i(a)(c_a^{m_a-1})
a\cdots\sigma^{i-1}(a)} =0$.   But $f(\ov{(c_a)^{m_a-1}a\cdots\sigma^{i}(a)})=f(\ov{\sigma^i(a)(c_a^{m_a-1})
a\cdots\sigma^{i-1}(a)})$,  a contradiction.

			We now prove (4).   Suppose that $\ov{c_a^{m_a}}$ is a non-zero element in $\soc(A)$.
It suffices to show  $\ov{c_{\sigma(a)}^{m_a}}\in\soc(A)$.  
First note that, using $f(xy)=f(yx)$, for any $x, y \in A$, we
see that $f(\ov{c_a^{m_a}})=f(\ov{c_{\sigma(a)}^{m_a}})$.
Hence $\ov{c_{\sigma(a)}^{m_a}}\ne 0$.  
 Suppose for contradiction that $\ov{c_{\sigma(a)}^{m_a}}\not\in\soc(A)$,  then
 $\ov{ a c_{\sigma(a)}^{m_a}}\ne 0$ since $a$ is
the only arrow such that $\ov{a\sigma(a)}\ne 0$.
But
\[a{c_{\sigma(a)}^{m_a}}=c_a^{m_a}a\]
and hence  $\ov a\ov{c_{\sigma(a)}^{m_a}}= 0$ since
$\ov{c_{\sigma(a)}^{m_a}}$ is in the (left) socle
of $A$, a contradiction.

Part (5) follows since $f(xy)=f(yx)$,
\[c_a^{m_a}=(a\cdots \sigma^{i-1}(a))(c_{\sigma^i(a)}^{m_a-1}\sigma^i(a)
\cdots\sigma^{|O|-1})\text{ and}\]\[
c_{\sigma^i(a)}^{m_a}=(c_{\sigma^i(a)}^{m_a-1}\sigma^i(a)\cdots\sigma^{|O|-1})
(a\cdots \sigma^{i-1}(a)).
\]

\end{proof}

\begin{lemma}\label{lem-q}Let $A=K\cQ/I$ be an indecomposable 
symmetric special
multiserial $K$-algebra 
 and let $f\colon A\to K$ be a  non-degenerate symmetric $K$-linear form 
 %such that $f(xy)=f(yx)$, for all $x,y\in A$ and 
 such that $\ker(f)$ contains no two-sided
ideals in $A$.  Let $e$ be a primitive idempotent in $A$ and let $p$ and $p'$ be nonzero elements in $e\soc(A)e$ such that
$f(p)=f(p')$.  Then $p=p'$.

\end{lemma}
\begin{proof}  Since $\ker f$ contains no non-zero two-sided ideals and since $\dim_K(e\soc(A)e)=1$,  we have that $f$ restricted to $e\soc(A)e $
is an isomorphism.  The result follows.
\end{proof}

%The next result will be used in Section \ref{sec-sma}.
 For the next  result we need to assume that the field $K$ is algebraically closed. 
Keeping the notations as above, we have the following. 

\begin{prop}\label{prop-good-pi} Let $K$ be an algebraically
closed field, let $A  $ be a basic  indecomposable
symmetric special multiserial  $K$-algebra, and let $\cQ$ be the quiver of $A$.
Then there exists a surjection $\pi\colon K\cQ\to A$ such 
that \begin{enumerate}
\item $\ker(\pi)$ is admissible, and 
\item   if $a$ and $b$ are arrows starting at a vertex $v$ in $Q$, then
$\pi(c_a^{m_a})=\pi(c_b^{m_b})$.
\end{enumerate}
\end{prop}

\begin{proof} Since $A$ is assumed to be finite dimensional and basic, there exists a surjection  $\pi'\colon K\cQ\to A$ such that
$\ker(\pi')$ is admissible.  Let $f\colon A\to K$ be a  non-degenerate symmetric  $K$-linear
form with no two sided ideal in its kernel. 
%Let $a$ and $b$ be arrows starting at vertex $v$. 
We now construct a
surjection $\pi\colon K\cQ\to A$ by defining, for each arrow $a$ in $\cQ$, a non-zero constant  $\lambda_a \in K$
such that  by setting $\pi(a)=\lambda_a\pi'(a)$  the desired properties hold. Since  $\ker(\pi') $ is admissible, clearly  $\ker(\pi)$ is admissible. 

We show that (2) holds one $\sigma$-orbit at a time.
Let $O$ be a $\sigma$-orbit and $a\in O$.  Fix a nonzero
element $k\in K$ and consider the cycle $c_a$.  Then by Lemma \ref{lem-orbits}  (1) and (3), $\pi'(c_a^{m_a})\in e_v\soc(A)e_v$
where $v$ is the vertex at which the arrow $a$ starts and $e_v$ is the associated primitive idempotent in $A$.   We know that $f( \pi'(c_a^{m_a}))\ne 0$.
Let $\lambda_a=(\frac{k}{f( \pi'(c_a^{m_a}))}) ^{1/m_a}$.  Note
that if we  set $\pi(a) = \lambda_a\pi'(a)$ and  $\pi(\sigma^i(a)) = \pi'(\sigma^i(a))$, for $1\le i\le |O|-1$, then $f(\pi(c_a^{m_a}))=k$.
%\ed{this looks okay -- your changes I mean}
%change $\pi'(a)$ to be $\lambda_a\pi'(a)$ and keep $\pi'$ unchanged
%on all other arrows, then $f(\pi'(c_a^{m_a}))=k$.  
By Lemma \ref{lem-orbits}(5),
 $ f(\pi(c_{\sigma^i(a)}^{m_a}))=k$ for $1\le i\le |O|-1$.  
That is, we define $\lambda_b=1$ if $b\in O$ and $b\ne a$. 

Let $\pi\colon\cQ\to A$ be the resulting surjection from the construction
above carried out for every $\sigma$-orbit.  Then we have that for each arrow $a$ in $\cQ$,
$f(\pi(c_a^{m_a}))=k$.  Applying Lemma \ref{lem-q}, we get the
desired result.
\end{proof}

%%%%%%%%%%%%%%%%%%%%%%%%%%%%%%%%%%%%%%%%%%%%%%%%%%%%%%%%%%%%%

%\section{Special multiserial algebras}\label{sec-sma}
\subsection{ Proof of Theorem~\ref{thm-sma}}

We are now able to prove Theorem~\ref{thm-sma}, which states that a \bca\ is a special multiserial algebra and conversely, that every symmetric special multiserial algebra is a \bca . 

%\begin{prop}\label{bca-sma} Every \bca\   is a special multiserial algebra.
%\end{prop}

%\begin{proof}Let $\Lambda$ be a \bca.  Consider Lemma 2.2 in \cite{GS}. {\Sib {\tt Here we
%have to write a little more since this is now in the other paper}}. 
%Taking $x$ and $y$ to be arrows, we see that condition (M) holds.
%\end{proof}

 %If $A$ is an indecomposable symmetric special multiserial algebra with quiver $\cQ$ and
%$\rad(A)^2\ne 0$,  we
%will let $\sigma\colon \cQ_1\to \cQ_1$ be defined by
%$\sigma(a)$ is the unique arrow such $\bar a\overline{\sigma(a)}\ne 0$.  Since
%$\sigma$ is a set isomorphism (see proof of Proposition \ref{strongM}), $\cQ_1$ is partitioned into the orbits of $\sigma$, which
%we write as $\mathbb O=\{O_1,\dots,O_m\}$.   These orbits are studied in more
%detail in the Appendix.  Note that if $O$ is an orbit in $\mathbb O$ and $a$ is an arrow in $O$,
%then $O=\{a,\sigma(a),\sigma^2(a),\dots,\sigma^n(a)\}$ where the cardinality of $O$, $|O|$,
%is $n+1$.    Furthermore, by Lemma \ref{lem-orbits}, there is a positive integer $m_a$
%such that $c_a=a\sigma(a)\cdots\sigma^{|O|-1}(a)$ is a cycle and $\pi(c_a^{m_a})$ is
%a nonzero element of  $ \soc(A)$. }

 %We now prove the converse to Proposition \ref{bca-sma}.

\emph{Proof of Theorem~\ref{thm-sma}.} 
%{\color{green} Do we need to assume $A$ basic? Also at the first reading I'm not sure that his proof is complete like this, but I have not had time to read it again. In order for you to be able to get the paper, I'll just hand it over to you as is and read it again when it comes back to me. }  \ed{In the statement of 3.6 we assume that $A$ is basic so there 
%is no problem ... also if $A$ is morita equiv to a basic alg (as in the definition of special multiserial algebra), that is no problem since in the morita class of an algebra, the basic alg is unique
%up to isomorphism. Thus, if you want we can change the statement to a spec. multiser. alg
%is morita equivalent to a \bca.}
 First assume that $A= K \cQ / I$ 
 is symmetric special multiserial.  
 By Proposition \ref{prop-good-pi} we can assume that there is a surjection 
$\pi\colon K\cQ\to A$ with $I = \ker \pi$ such that if $a$ and
$b$ are arrows in $\cQ$ starting at the same vertex
then $\pi(c_a^{m_a})=\pi(c_b^{m_b})$ where $c_a, m_a, c_b, m_b$ are as defined in Lemma~\ref{lem-orbits}.  

  Let $\sigma\colon \cQ_1\to 
\cQ_1$ be the permutation induced by $I$.  For each $\sigma$-orbit $O$, choose an arrow $a\in O$, and let $L_O$ denote the multiset consisting
of the vertices occuring in $c_a$, counting repetitions.  More precisely,
if  $c_a =a\sigma(a)\cdots\sigma^{|O|-1}(a)$ and if $\sigma^i(a)$ is an arrow from
$v_{j_i}$ to $v_{j_{i+1}}$ then $L_O=\{v_{j_0},v_{j_1},\dots
v_{j_{|O|-1}}\}$.  
Note, for $i=0$, $\sigma^0(a)=a$ is
an arrow from $v_{j_0}$ to $v_{j_{1}}$ and $\sigma^{|O|-1}(a)$
is an arrow from $v_{j_{|O|-1}}$ to $v_{j_0}$; that is, $v_{j_{|O|}}=v_{j_0}$.  By construction,
the set  $L_O$ is independent of the choice of $a\in O$ since if $a'\in O$, then
$c_{a'}$ is a cyclic permutation of the arrows of $c_a$.

We now construct the desired \bca\ which we denote by $\Gamma$.
We begin with a set $\Gamma_0^* \subset \Gamma_0$ which is in one-to-one correspondence with the
set of $\sigma$-orbits $\cO=\{O_1,\dots, O_m\}$ of $\sigma$.  
We let $\Gamma_0^*=\{\alpha_1,\dots,
\alpha_m\}$ where $\alpha_i$ corresponds to $O_i$.  These will be the 
nontruncated vertices of $\Gamma$.  The polygons of $\Gamma$
are in one-to-one correspondence with the vertices of $\cQ$ such that if $\cQ_0=
\{v_1,\dots,v_n\}$ then we set $\Gamma_1=\{V_1,\dots, V_n\}$ where the polygon $V_i$ corresponds to the vertex $v_i$ of $\cQ$. 

We need to describe the truncated vertices in  $\Gamma_0$, the elements
that occur in each polygon $V_i$, the
multiplicity function $\mu$ and the orientation $\mathfrak o$.  We begin fixing
$\alpha\in \Gamma_0^*$,  $V\in \Gamma_1$ and determine how many times
$\alpha$ occurs as an element in $V$.   Suppose that $\alpha$ corresponds to
the $\sigma$-orbit $O$ and $V$ corresponds to $v\in\cQ_0$.  
Then $\alpha$ occurs in $V$ the number of times $v$ occurs in $L_O$.

Next we define the set of truncated vertices in $\Gamma$. 
For each polygon $V$ that  consists of exactly one nontruncated vertex, say $\alpha$, we
add a new vertex $\alpha_V$ to the vertex set $\Gamma_0$ and to $V$.  Thus $V=\{\alpha, \alpha_V\}$.
We set $\mu(\alpha_V)=1$ and hence $\alpha_V$ is a truncated vertex of $\Gamma$.
In this way, we have defined the truncated vertices  and we see that condition C3 is satisfied. We also see  that condition
C2, namely that  $|V|\ge 2$,  is satisfied.
From this construction it is clear that $\Gamma$ satisfies condition C4. 

For each  $\alpha\in\Gamma_0^*$, let $\mu(\alpha)=m_a$,  where $a$ is an arrow
in the $\sigma$-orbit corresponding to $\alpha$.   By Lemma \ref{lem-orbits}(4), $\mu$
is independent of the choice  of $a$. We have defined $\mu$ to be $1$ on truncated
vertices and hence, we have completed the definition of the multiplicity function $\mu$.

Finally,  we need to describe the orientation  $\mathfrak o$.  For this, we let
$\alpha$ be  a vertex in $\Gamma_0^*$ and assume that $\alpha$ corresponds
to the $\sigma$-orbit $O=\{a,\sigma(a),\sigma^2(a),\dots, \sigma^{|O|-1}(a)\}$.
 Then, as above,  $\sigma^i(a)$ is an arrow from $v_{j_i}$ to $v_{j_{i+1}}$, for $0\le i\le |O|-1$
and $v_{j_{|O|}}= v_{j_0}$.   If $V_{j_i}$ is the polygon corresponding to the vertex
$v_{j_i}$, then we let $V_{j_0}<\cdots<V_{j_{|O|-1}}$ be the successor sequence
at $\alpha$.   Varying $\alpha\in\Gamma_0^*$ yields an orientation $\mathfrak o$.

 We have now constructed a \bc \ $\Gamma = (\Gamma_0, \Gamma_1, \mu, \mathfrak o)$. Let $\Lambda$ be the \bca\ associated to $\Gamma$.
We show that  $\Lambda$ is isomorphic to $A$.  Let $\cQ_{\Lambda}$ be the quiver
of $\Lambda$.
We begin by showing that there is an isomorphism of quivers from $\cQ_{\Lambda}$ to $\cQ$.
By our construction  of the quiver of $\cQ_{\Lambda}$ given in
the beginning of this section, we see that the vertices of $\cQ_{\Lambda}$
correspond to the polygons in $\Gamma$, which in turn correspond
to the vertices of $\cQ$.  Thus, we get a one-to-one correspondence
between the vertices of $\cQ_{\Lambda}$ and the vertices of  $ \cQ$.
Again by our construction of the quiver of $\cQ_{\Lambda}$ given in
the beginning of this section, the arrows of $\cQ_{\Lambda}$ correspond
to successors in the successor sequences of $\Gamma$.   But the successor
sequences in $\Gamma$ correspond to the $\sigma$-orbits and 
each arrow in $ \cQ$ occurs once, in exactly one $\sigma$-orbit.
Thus, the quivers  $\cQ_{\Lambda}$ and $\cQ$ are isomorphic.

An isomorphism from $\cQ_{\Lambda}$ to $\cQ$ induces an isomorphism
of path algebras $K\cQ_{\Lambda}\to K\cQ$.
Thus, we obtain a surjection $\varphi\colon K\cQ_{\Lambda}\to A$.
It is straight forward to see that relations of types  one, two and three 
are all in $\ker(\varphi)$.  Hence $\varphi$ induces a surjection
from $\Lambda$ to $A$.   To complete the proof, we consider
the uniserial  modules in both algebras, that is in the \bca\ $\Lambda$ and in $A$.  We now apply the
results from Section 3 in \cite{GS} on uniserial modules in a \bca \ and the results
from Section 2 of this paper on uniserial modules in a  symmetric special multiserial algebra. It follows that the uniserial $\Lambda$-modules $U$ 
such that  $U$ is not a projective $\Lambda$-module and
such that $U$ is maximal with this property,  correspond to the 
uniserial $A$-modules $U'$  such that  $U'$ is not a projective  
 $A$-module and such that 
$U'$ is maximal with this property.
  Thus 
the dimensions of $\rad(\Lambda)$ and $\rad(A)$  are equal.
It follows that the surjection from $\Lambda$ to $A$ is an
ismorphism and we are done. 

 The converse immediately follows from Proposition 2.8 in \cite{GS}.

%\ed{To prove the converse, we recall Lemma {\color{green}{\tt put in correct
%number}} in \cite{GS}. 
%}

%\ed{
%\newtheorem*{lemma*}{Lemma}
%\begin{lemma*}\cite{GS} Let $\Gamma$ be a Brauer configuration with associated
%Brauer configuration algebra $\Lambda = \K \cQ /I$ and let $V\in\Gamma_1$,  $\alpha\in\Gamma_0$  a 
%nontruncated vertex
%in $V$, and $C=a_1a_2\dots a_{\val(\alpha)}$ a $(V,\alpha)$-cycle.
%Let $p=a_1a_2\dots a_j$, for some $1\le j\le \val(\alpha)-1$. Set 
%$x=C^sp$ and $y=p{C'}^s$, for some $ 0 \le  s<\mu(\alpha)$, where
%$C'=a_{j+1}\dots a_{\val(\alpha)}a_1\dots a_j$.
%Then\begin{enumerate}
%\item  $a_i\ne a_j$,  for $i\ne j$.
%\item $\bar x\ne 0$.
%\item If $a$ is an arrow in $\cQ$, then $\overline{xa}\ne 0$ if and only
%if $a=a_{j+1}$.
%\item $\bar y\ne 0$.
%\item If $a$ is an arrow in $\cQ$, then $\overline{ay}\ne 0$ if and only
%if $a=a_{\val(\alpha)}$.
%\end{enumerate}
%\end{lemma*} 
% In the notation of this Lemma, 
%taking the paths $x$ and $y$ to be arrows, we see that condition (M) holds.}
\hfill $\Box$

\section{Symmetric  algebras with radical cube zero are special multiserial}\label{radicalcubezero}

 In this section we show that the class of special multiserial algebras contains another 
class of well-studied algebras. Namely that of symmetric algebras with radical cube zero. 
We show that  basic symmetric algebras with radical cube zero are special multiserial and hence  that they are Brauer configuration algebras.  We remark that in \cite{GS}, it is proved
that the class of 
symmetric algebras  with radical cube zero associated to a symmetric matrix with non negative integer coefficients is
the same as the class of \bca s in which the polygons have no repeated vertices.  Our main results
of this section show
that dropping this restriction on polygons classifies all symmetric  algebras  with radical cube zero.

More precisely, we prove the following result.

\begin{thm}\label{thm-r3equiv} Let $K$ be an algebraically closed field and
let $A\cong K\cQ/I$ be a finite dimensional  indecomposable $K$-algebra.  Assume that  $\rad^3(A)=0$ but $\rad^2(A)\ne 0$.
Then the following statements are equivalent.
\begin{enumerate}
\item $A$ is a  symmetric $K$-algebra.

\item $A$ is a symmetric multiserial $K$-algebra.
\item $A$ is a symmetric special multiserial $K$-algebra.
\item $A$ is isomorphic to  an indecomposable \bca.

\end{enumerate}
\end{thm}

 \noindent {\it Proof.} Clearly (2) implies (1).  We see that (3) implies (2) by
Theorem \ref{multiserial theorem}.   By Theorem \ref{thm-sma}, (3) holds if and only if
(4) holds.  It remains to show that (1) implies (4).  For this we start with some
preliminary results, culminating in Theorem \ref{thm-r30}, which is the
desired result. \hfill $\Box$

%\vskip .2in

For the remainder of this section, we let
 $A = K\cQ /I $ be an indecomposable symmetric $K$-algebra such that
$\rad^3(A)=0$ but $\rad^2(A)\ne 0$.
% and every simple  $A$-module is one dimensional.
   We assume that $ I$ is an admissible
ideal and let
$f \colon A\to K$ be a non-degenerate symmetric  linear form 
such that $\ker(f)$ does not contain a two-sided ideal of $A$.  If $M$ is a right
$\Lambda$-module,  then the \emph{Loewy length} of $M$ %, denoted $  \ell\ell(M)$,
is $n$ if $M\rad^{n-1}(A)\ne 0$ and $M\rad^{n}(A)= 0$.  We fix a surjection $\pi\colon K\cQ\to
A$ with kernel $I$ and if $x\in K\cQ$, we will write $\bar x$ for
$\pi(x)$.   More generally, we will write $\bar a$ for elements in $A$.

% If $\pi\colon K\cQ\to
%A$ is a surjection with kernel $I$ and if $x\in K\cQ$, we will write $\bar x$ for
%$\pi(x)$.   More generally, we will write $\bar a$ for elements in $A$.

The next result is well-known but we include
 a proof for completeness.

\begin{lemma}\label{lem-ll} Keeping the notations and assumptions
above, 
every indecomposable  projective $A$-module has Loewy
length $3$.
\end{lemma}

\begin{proof} 
%Fix a surjection $\pi\colon K\cQ\to A$ where $\cQ$
%is the quiver of $A$.

 Let $P$ be an indecomposable projective $A$-module
of Loewy length $2$. Then $P$ is the projective cover of a simple
$A$-module $S$. Since $A$ is symmetric, $P$ has  simple top and simple socle isomorphic to $S$ and these are the only composition factors of $P$. Thus $P$ is an extension 
of $S$ by $S$ and if $S$ is the simple at vertex $v$ in $\cQ$ then there is a loop at $v$ in $\cQ$. Furthermore, there is no other arrow leaving $v$. 
Since $A$ is symmetric, $P$ is also the injective hull of $S$ and so there is no arrow entering $v$. Thus there is a loop at $v$ and no other arrow entering or leaving $v$. Thus there is a factor $K[x]/(x^2)$ of $A$,  
contradicting the indecomposability of $A$.
% If $P$ is an indecomposable projective $A$-module
%of Loewy length $2$ and $P$ is the projective cover of a simple
%$A$-module $S$, with $S$ the simple module corresponding to
%the vertex $v\in\cQ$,  then, since $A$ is symmetric, $P$ is an extension
%of $S$ by $S$ {\Sib I don't understand why $P$ is an extension of $S$ by $S$? Is it not an extension of $S$ by $\rad(S)$?}.  So there is a loop at $v$ in the quiver $\cQ$ and there
%are no other arrows leaving $v$.   \ed{If there is an arrow from a vertex
%$w$ to $v$, $w\ne v$, then since $A$ is symmetric and $S$ is
%a summand of $\rad({\Sib e_w}A)/\soc({\Sib e_w}A)$} {\Sib {\tt This phrase is not complete}}.   Since $\soc({\Sib e_w}A)$ is a simple
%module associated to $w$, there must an arrow from
%$v$ to $w$ which we have seen is impossible.  Thus there
%is a loop at vertex $v$ and no other arrows entering or leaving $v$.
%This contradicts the indecomposability of $A$.
\end{proof}

%Let $\cQ$ be the quiver of $A$, $\pi\colon K\cQ\to A$ a $K$-algebra surjection
%with kernel $I$.  Let $f\colon A\to K$ be a linear form such that
%$f(\ov{xy})=f(\ov{yx})$ and $\ker(f)$ contains no two-sided ideals in $A$.
%The next lemma relates the kernel of $f$ to $I$.

\begin{lemma}\label{lem-kerq}  Keeping the notations and assumptions
above, let  $e_v$ be the idempotent at   a vertex $v$ in $K\cQ$ and let $x$ be a linear combination
of  paths of length $2$ such that $e_vxe_v=x$.  Then $x\in I$ if and only if
$f(\bar x)=0$.
\end{lemma}

\begin{proof}
Suppose $x\not\in I$.  Then  $ \bar x$ is a nonzero
element of the socle of $A$.  If $f( \bar x)=0$ then the $K$-span
of $\bar x$ is in $\ker(f)$.  Since $\soc(A)$ is semisimple and each
simple $A$-module is one dimensional, we obtain a two sided
ideal in $\ker(f)$ which is a contradiction.  Hence $f(\bar x)\ne 0$.
Next suppose that $x\in I$.   Then $\bar x=0$ and we
conclude that $f(\bar x)=0$.
\end{proof}

Our next lemma shows the special nature of 
symmetric algebras  with radical cube zero.

\begin{lemma}\label{lem-1side}  Keeping the notations and assumptions
above, 
let $a$ and $b$ be arrows in $\cQ$.  
%If $x\in K\cQ$ we denote by $\bar x$ the image of $x$ under the canonical surjection from $K\cQ$ to $K\cQ/I$. 
 Then the following statements
are equivalent:
\begin{enumerate}
\item $ab\not\in I$.
\item $ba\not\in I$.
\item $\ov{ab}$ is a nonzero element of $\soc(A)$.
\item $\ov{ba}$ is a nonzero element of $\soc(A)$.
\item $f(\ov{ab})\ne 0$.
\item $f(\ov{ba})\ne 0$.

\end{enumerate}
%In particular,
%if $a$ is an arrow in $\cQ$ having exactly one
%arrow $b$ with  the property that $ab\not\in I$, then $b$ is the unique arrow
%such that $ba\not\in I$.
\end{lemma}

\begin{proof} Note that since $A$ is  symmetric, if
$a$ and $b$ are arrows with $\ov{ab}\in\soc(A)$ and $\overline{ab} \neq 0$, then
there is a vertex  $v$ such that  $e_vabe_v=ab$ where $e_v$ is the corresponding idempotent in $K\cQ$. 
Using Lemma \ref{lem-kerq} and that  $f(\ov{ab})=f(\ov{ba})$,  it is clear that parts (1), (2), (5), and (6) are equivalent. Using that    
$\ker(f)$ cannot  contain any non-zero two-sided ideals and that $\rad^3(A)=0$,
we obtain that  part (3) is equivalent 
to part (1) and part (4) is equivalent part (2).
\end{proof}

The next result shows that  in general, for   a  basic indecomposable symmetric $K$-algebra $A$ such that
$\rad^3(A)=0$ but $\rad^2(A)\ne 0$  there is a special way of presenting $A$ as $K\cQ/I$.
For this we define a set $Arr$ whose $K$-span equals the $K$-span of  the image of the arrows in $\cQ$ and such that $Arr$ satisfies a tight set of multiplicative properties. 

\begin{prop}\label{prop-spec}Let $K$ be an algebraically closed
field and let $A$ be a basic  indecomposable symmetric $K$-algebra such that
$\rad^3(A)=0$ but $\rad^2(A)\ne 0$.

Then there is a $K$-linearly independent set $Arr\subset \rad(A)$  with the following 
properties.
\begin{enumerate}
\item $Arr$ generates  $\rad(A)$ as a two-sided ideal. 
\item If $\bar x$ is a nonzero linear combination of elements in  $Arr$ then
$\bar x\not\in \rad^2(A)$.
\item If $\bar a\in Arr$ is such that $ \bar a^2\ne 0$, then $\ov{ab}=0$ for all
$\bar b\in Arr$, $\bar b\ne \bar a$.
\item If $\bar a,\bar b\in Arr$ with $\bar a\ne \bar b$ and $\ov{ab}\ne 0$, then $\ov{ac}=0=\ov{bc}$ for
all $\bar c\in Arr$ with $\bar c\ne \bar a,\bar b$.
%\item For each pair $\bar a,\bar b$ satisfying (4), one of the pair is called a \emph{chosen element}.
\item For each $\bar a,\bar b\in Arr$, not necessarily distinct, if $\ov{ab}\ne 0$ then
$f(\ov{ab})=1$.
\end{enumerate}
\end{prop}

\begin{proof}The assumption that the simple  $A$-modules are one
dimensional implies there is a surjection $\pi\colon K\cQ\to A$ such that
$\cQ$ is the quiver of $A$ and that $\ker(\pi) $  is an admissible ideal. 
Let $f \colon A\to K$ be a linear form obtained from $A$ being symmetic.

Let $Arr=\pi(\cQ_1)$.  It follows that $Arr$  is linearly independent over $K$ and generates $\rad(A)$.  Moreover $Arr$ satisfies property (2).
If $\bar a$ and $\bar b$ are in $Arr$, then we set $\gamma_{a,b}=f(\ov{ab})$.

Let $Y=\text{Span}_K(Arr)$. Note that $Arr$ is a $K$-basis of $Y$.
We begin by  making a series of  linear changes of bases starting with
the basis $Arr$ of $Y$.
Suppose that there is an element $\bar a\in Arr$ such that $\bar a^2\ne 0$.
Then consider the change of basis with $\bar a$ remaining unchanged and
if $\bar b\in Arr$ with $\bar b\ne \bar a$, replace $\bar b$ by $\bar b-\frac{\gamma_{a,b}}{\gamma_{a,a}} \bar a$.
Note that after this change of basis, if $\bar b\ne \bar a$, 
$f(\bar a (\bar b-\frac{\gamma_{a,b}}{\gamma_{a,a}}\bar a))=0$ and hence, in the
new basis, $\ov{ab}=0$ by Lemma \ref{lem-1side}.  By abuse of notation,
we still call the new basis $Arr$.

If there is another $\bar b\in Arr$ such that  $\bar b^2\ne 0$, perform the same change
of basis for $\bar b$ instead of $\bar a$.  
Note that under this change of basis,  $\bar a$ remains  unchanged since
$\gamma_{a,b}=0$.
 Continuing in this fashion, we arrive at a basis, again called $Arr$, such that 
if $\bar a$ is an element in  $Arr$ and $\bar a^2\ne 0$, then, for all $\bar b\ne \bar a$,  we have
$\ov{ba}=0$ and $\ov{ab}=0$.

Now let $\bar a$ be an element in $Arr$ with $ \bar a^2=0$.  Then   by Lemma \ref{lem-ll} there  must be
an element $\bar b$ in $Arr$, $\bar b\ne \bar a$, such that $\ov{ab}\ne 0$.  Note that
$\bar b^2=0$, since if not, $\ov{ab}$ would equal 0.   Consider the change of
basis that leaves $\bar a$ and $\bar b$ unchanged, and where if $\bar c$ is an element in $Arr$ different
from $\bar a$ and $\bar b$,  we replace $\bar c$ by $\bar c-\frac{\gamma_{b,c}}{\gamma_{a,b}}\bar a
-\frac{\gamma_{a,c}}{\gamma_{a,b}}\bar b$.   Applying $f$ to the new basis,
we see that $\ov{ab}\ne 0$, $\ov{ac}=0=\ov{bc}$ for all $\bar c$ different from
$\bar a$ and $\bar b$.   Note that if $\bar c$ is an element of $Arr$ with $\bar c^2=0$,  then
$\bar c$ remains unchanged since $\gamma_{a,c}=0=\gamma_{b,c}$.
Continuing in this fashion, we obtain a new basis of $Y$,  which we call again $Arr$ satisfying  properties (3) and (4).
%\begin{enumerate} \setcounter{enumi}{2}
%\item If $\bar a\in Arr$ is such that $\bar a^2\ne 0$, then $\ov{ab}=0$ for all
%$\bar b\in Arr$, $\bar b\ne\bar a$.
%\item If $\bar a,\bar b\in Arr$ with $\bar a\ne \bar b$ and $\ov{ab}\ne 0$, 
%then $\ov{ac}=0=\ov{bc}$ for
%all $\bar c\in Arr$ with $\bar c\ne \bar a,\bar b$.
%\end{enumerate}

For each pair $\bar a,\bar b$ satisfying (4) above,  
choose either $\bar a$ or $\bar b$ and call it
a \emph{chosen element}.
We make one final change of basis of $Y$.
For each $\bar a\in Arr$ such that $\bar a^2\ne 0$,
replace $\bar a$ by $(1/(\gamma_{a,a})^{\frac{1}{2}})\bar a$.
For each   pair $\bar a,\bar b$  satisfying (4) above,
replace the chosen element, say  $\bar a $, by 
$(1/\gamma_{a,b})\bar a$ and leave $\bar b$ unchanged.
  We then  obtain a basis of $Y$, which   we call again $Arr$, satisfying properties (1)-(5)
%\begin{enumerate}
%\item If $\bar a\in Arr$ is such that $\bar a^2\ne 0$, then $\ov{ab}=0$ for all
%$\bar b\in Arr$, $\bar b\ne \bar a$.
%\item If $\bar a,\bar b\in Arr$ with $\bar a\ne \bar b$ and $\ov{ab}\ne 0$, then $\ov{ac}=0=\ov{bc}$ for
%all $\bar c\in Arr$ with $\bar c\ne \bar a,\bar b$.
%%\item For each $\bar a,\bar b$ satisfying (2), one is called a chosen element.
%\item For each $\bar a,\bar b$ in $Arr$, not necessarily distinct, if $\ov{ab}\ne 0$ then
%$f(\ov{ab})=1$.
%\end{enumerate}
and we  take this to be the desired set. 
\end{proof}

 We remark that  it follows from  the proof of Proposition 5.5 that the canonical surjection $\pi : K\cQ \to A$ maps the arrows of $Q$ bijectively to $Arr$. 

We now present the final result needed to finish the proof of Theorem \ref{thm-r3equiv}.

\begin{thm}\label{thm-r30}
Let $K$ be an algebraically closed field 
and let $A=K\cQ/I$ be a  finite dimensional  basic indecomposable $K$-algebra. Suppose that $A$ is symmetric and  that
$\rad^3(A)=0$ but $\rad^2(A)\ne 0$.  
%We assume that $ I$ is an admissible ideal.  
Then $A$ is isomorphic to a \bca.
\end{thm}

\begin{proof}
 Let $Arr$ be a set satisfying properties (1)-(5) in Proposition~\ref{prop-spec}.  Recall that if $x\in K\cQ$ then $\bar x$ will denote the
image of $x$ in $A$ under the canonical surjection $K\cQ \to A$
where the arrows of $\cQ$ are mapped bijectively to $Arr$. We show that property (M) holds. 
Let $a$ be an arrow in $\cQ$  and $\bar a$ its image in $Arr$.    First we show that there is at most one arrow $b$ such that
 $ab\notin I$.   If $a^2\notin I$, that is if $\bar{a}^2 \neq 0$, then Proposition \ref{prop-spec}(3) yields the result.
If $a^2\in I$, that is, if  $\ov{a}^2=0$, then Proposition \ref{prop-spec}(4) shows that
if $ab\notin I$ for some arrow $b$, 
then $ad\in I$ for all $d\ne a$ or $b$.  But  $a^2\in I$ and hence there is at
most one arrow  $b$ such that  $ab\notin I$.

Given an arrow $a$ of $\cQ$, by Lemma~\ref{lem-1side} we have that $ab \notin I$ for some arrow $b$ if and only if $ba \notin I$. By the first part of the proof above $b$ is unique if it exists. Therefore, it follows directly that if $ab \notin I$ then $ba \notin I$ and there is no other arrow $c$ with
$c \neq b$ such that $ca \notin I$.  Note that this also holds if $b$ is equal to $a$.  

Thus $(M)$ holds and hence $A$ is a symmetric special multiserial algebra. The result then follows from Theorem \ref{thm-sma}.

\end{proof}

\bibliographystyle{plain}

\end{document}